\documentclass[twoside,11pt]{article}

\usepackage{jmlr2e}
\graphicspath{{./fig/}}

\jmlrheading{}{}{}{}{}{}{Fengmiao Bian, Ren Liu and Xiaoqun Zhang}

\ShortHeadings{A Stochastic three-block algorithm and its applications to QDNNs}{Fengmiao Bian, Ren Liu and Xiaoqun Zhang}
\firstpageno{1}

\begin{document}

\title{A stochastic three-block splitting algorithm and its application to quantized deep neural networks}

\author{\name Fengmiao Bian \email mafmbian@ust.hk\\
       \addr Department of Mathematics\\
       The Hong Kong University of Science and Technology\\
       Clear Water Bay, Kowloon, Hong Kong, CHINA
       \AND
       \name Ren Liu \email liur0810@sjtu.edu.cn \\
       \addr School of Mathematical Sciences\\ 
       Shanghai Jiao Tong University\\
        Shanghai 200240, CHINA
        \AND
       \name Xiaoqun Zhang \email xqzhang@sjtu.edu.cn \\
       \addr School of Mathematical Sciences, MOE-LSC \\
       and Institute of Natural Sciences\\
       Shanghai Jiao Tong University\\
       Shanghai 200240, CHINA}

\editor{}

\maketitle

\begin{abstract}
Deep neural networks (DNNs) have made great progress in various fields. In particular, the quantized neural network is a promising technique making DNNs compatible on resource-limited devices for  memory and computation saving. In this paper, we mainly consider a non-convex minimization model with three blocks to train quantized DNNs and propose a new stochastic three-block alternating minimization (STAM) algorithm to solve it.  We develop a convergence theory for the STAM algorithm and obtain an $\epsilon$-stationary point with optimal convergence rate $\mathcal{O}(\epsilon^{-4})$. Furthermore, we apply our STAM algorithm to train DNNs with relaxed binary weights. The  experiments are carried out on three different network structures, namely VGG-11, VGG-16 and ResNet-18. These DNNs are trained using two different data sets, CIFAR-10 and CIFAR-100, respectively. We compare our STAM algorithm with  some classical efficient algorithms for training quantized neural networks. The test accuracy indicates the effectiveness of STAM algorithm for training relaxed binary quantization DNNs.
\end{abstract}


\section{Introduction}\label{sec:Intro}

\subsection{Background}\label{background}
Deep neural networks (DNNs) have become important tools for complex data learning and achieved great success in many practical fields, such as computer vision, speech recognition, automatic driving and so on \citep{SL, RHGS, KSH, SZ, WWGL}. This great success mainly depends on the flexibility of neural networks and  their complex nonlinear structure. At present, most research  makes neural networks more flexible  by increasing  number  of layers and/or width of neural networks  \citep{GBC, S}. However, these also lead to a large increase of the number of parameters in neural networks. Since most parameters are floating point numbers, one requires a lot of storage space. For example,  the AlexNet Caffemodel is over 200MB, and the VGG-16 Caffemodel is over 500MB \citep{HMD}. It makes these powerful  networks  very difficult to embed into portable devices, such as mobile phones, laptops and so on. 

Therefore, many literature are devoted to compressing DNNs to reduce the storage space of parameters while maintaining the test accuracy of network as much as possible. A common method of compressing neural networks is quantization weights, that is to say, the network is compressed by reducing the bit number of  parameters.  Many research work are devoted to  this direction. For example, in the literature \citep{HMD}, they reduced the parameters of the full connection layer to 5 bits and of the convolution layer to 8 bits, while ensuring almost the same accuracy. In \citep{LCMB, D, LTA}, the parameters in the network pretraining are quantified as 4-12 bits. In \citep{LS}, the authors also proposed two algorithms to quantify each layer of neural network sequentially in a data-dependent manner, and then they trained DNNs on MNIST and CIFA10 datasets with relaxed ternary or 4-bits weights. 

The lowest bit quantization is the 1-bit quantization which is called the binary quantization. For binary quantized DNNs, both weights and activations can be expressed as -1 (0) or 1, thus not taking up too much memory. In addition, by binary quantization, heavy matrix multiplication can be replaced with lightweight bitwise and bit-counting operations. Therefore, binary DNNs achieve significant acceleration during the inference and save memory and power consumption. These clearly make binary DNNs hardware-friendly. Indeed, there have been some works to verify the effectiveness of binary DNNs, such as BNN \citep{HCSYB} and Xnor-Net \citep{RORF}. In particular,  Xnor-Net  with 1-bit convolution layer can achieve 32$\times$ memory saving and 58$\times$ faster convolutional operations. Furthermore, 1-bit quantization DNNs have also been successfully adopted in some practical applications, such as computer Vision \citep{LBH, SZ, HZRS}, Image classification \citep{ZYYH, LWLYLC, GLJLHLYY}, Detection \citep{LWLQYF} and so on.

From another point of view, the binary quantization problem has been considered formulating as a non-convex optimization problem. Thus, many studies focus on optimized binary quantization which minimizes quantization error and improves the loss function. For example, in \citep{CP, RVV}, they used the classical projection stochastic  gradient descent (PSGD) algorithm to train DNNs with binary quantized weights. Further, in \citep{CBD} they introduced BinaryConnect (BC) which is a modification of  the projection stochastic gradient descent algorithm to train a DNN with binary weights, while the experiments showed that the accuracy was significantly improved. In \citep{LDXSSG} the authors studied the BC method from a theoretical perspective and established the convergence of BC under the strongly convex assumptions. In \citep{LLZJ}, the quantitative neural network is regarded as an optimization problem with constraints, and the authors decoupled the continuous variables from the discrete constraints of the neural network based on ADMM. In \citep{YZLOQX}, they proposed a relaxed binary quantization algorithm called BinaryRelax (BR) to better solve non-convex optimization problems with discrete constraints. The algorithm is  a two-stage method with a pseudo-quantization weight constantly close to the quantization weight by increasing regularization parameters at the first stage, while in the second stage the quantized weights are directly adopted.   

\subsection{Problem formulation and motivation}\label{algorithm}
Previous work focused on the algorithm of quantization weight, which all directly minimized the loss function on quantization weight.  The non-convexity of this kind of optimization problem is rather strong and results in many  local minima, which makes the algorithms  get stuck at some ``bad" local optimum.

In this paper, we attempt to explore new formulations that reveal the relationship between quantized  and floating-point weights. It is a natural idea that quantized parameters should approximate the full-precision parameter as closely as possible,  expecting that the performance of the binary neural network model will be close to the full-precision one. Thus, we construct the following new model for training DNNs with quantized weights:
\begin{equation}\label{binary-model}
\min_{W,  \tilde{W}} \frac{\lambda}{2} \| W - \tilde{W} \|_{F}^2 + L_{W} (p, q) + \mathcal{I}_{\mathcal{Q}}(\tilde{W}). 
\end{equation}
Here  $W$ is the float parameters in the neural network and $\tilde{W}$ is the corresponding quantized parameters. $L_{W}(p, q)$ is the loss function of the neural network, $p$ is the input data of the neural network and $q$ is the corresponding label. The function $\mathcal{I}_{\mathcal{Q}}$ denotes the indicator function of the set of quantized weights $\mathcal{Q}$. In the model \eqref{binary-model}, we use the loss function $L_{W}(p, q)$ to find the floating point parameters with good generalization performance. Further,  under the interaction of $\| W - \tilde{W} \|_{F}^2$ and $\mathcal{I}_{\mathcal{Q}}(\tilde{W})$  the quantized parameters would be close enough to the floating parameters. We would design a new stochastic three-block alternating minimization algorithm to solve the problem \eqref{binary-model}. The problem can be formulated in a more general form with three-block composite structure: 
\begin{equation}\label{model}
\min_{x,y} \Phi(x,y) := F(x) + G(y) + H(x,y),
\end{equation}
where $F, H, G$ are proper lower semi-continuous functions. Here we emphasize that the functions $F$, $G$, and $H$ are not necessarily convex. In deep neural networks, the function $G(y)$ is generally a loss function which is a sum of many terms, such as $G(y)=\frac{1}{N} \sum_{i=1}^N G_i (y)$ with $N$ being large. 

\subsection{The proposed algorithm and related work.}
The main idea of our proposed algorithm is to  minimizing variables $x$ and $y$ alternately for solving the optimization problem \eqref{model}. For $y$-direction, as the function $G$ has a large-sum structure, we consider to linearize the function $G+H$  and utilize stochastic gradient estimators instead of full gradient calculations. For $x$-direction,  the corresponding composite problem is solved using the Douglas-Rachford splitting method. 
Based on the above ideas, we propose a stochastic three-block alternating algorithm to solve the non-convex problem \eqref{model}; see Algorithm \ref{alg:STAM}. {\it Throughout} the paper, we assume that the gradient estimator $\tilde{\nabla} G(y)$  in Algorithm \ref{alg:STAM} is unbiased. More arguments of unbiased gradient estimators can be seen in subsection \ref{UGE01}.   

\begin{algorithm}[htp]
\caption{ A Stochastic Three-block Alternating Minimization (STAM) Algorithm}
\label{alg:STAM}
\begin{algorithmic}
\STATE{{\bf{Step 0.}}  Choose the parameters $\gamma,~\beta >0$ and an initial point $x^{0},~y^{0}.$}

\STATE{{\bf{Step 1.}}  For $t = 0, \dots, T-1$ and set
\begin{subequations}\label{algalg1}
\begin{align}
&y^{t+1} \in \arg\min_{y} \Bigg\{  \langle \tilde{\nabla} G( y^t) + \nabla_y H(x^t, y^t), y - y^t \rangle + \frac{\beta}{2} \| y - y^t \|^2 \Bigg\}, \label{algy}\\
&x^{t+1} \in  \arg\min_{x}  \Bigg\{  H(x, y^{t+1}) + \frac{1}{2\gamma}\|x-z^{t}\|^{2} \Bigg\}, \label{algx}\\
&u^{t+1} \in  \arg\min_{u} \Bigg\{ F(u) + \frac{1}{2\gamma}\| 2x^{t+1} -z^{t} - u  \|^{2} \Bigg\},  \label{algu} \\
&z^{t+1} = z^{t} + (u^{t+1} - x^{t+1}), \label{algz}
\end{align}
\end{subequations}
where $\tilde{\nabla} G(y)$ is a gradient estimator of $\nabla G(y)$.  
}
\vspace{0.2cm}
\STATE{{\bf{Step 2.}}  Output $\{ y^{t+1}, u^{t+1} \} $.}

\end{algorithmic}
\end{algorithm}

We point out that many algorithms have been designed to solve the three-block problem of form \eqref{model}. One of the most famous algorithms is the proximal alternating linearized minimization (PALM) algorithm proposed by \citep{BST} which alternately solves two linearized proximal problems.  In \citep{DY}, a three-operator splitting (called DYS) algorithm is formulated to solve the convex optimization problem of three terms without cross terms. This algorithm is also a generalization of Dogulas-Rachford Splitting algorithm \citep{LM}. See also non-convex Douglous-Rachford algorithm by \citep{LP}. Later, \citep{BZ} established convergence of this three-operator splitting algorithm in non-convex framework. In \citep{BCN}, the authors also proposed a proximal alternating algorithm to solve the three-term problem with linear operator composition. However, all of the algorithms mentioned above are deterministic. As a result, when solving large-scale optimization problems, deterministic algorithm will consume a lot of time, which damps the efficiency of the algorithm.  To avoid such difficulty, many stochastic algorithms for non-convex problem involving three terms have also been proposed.   In \citep{DTLDS}, the authors presented a stochastic proximal alternating linearized minimization (called SPRING) algorithm which combines a class of variance-reduced stochastic gradient estimators. In \citep{YMS}, the authors extended DYS algorithm to a stochastic setting where the unbiased stochastic gradient estimators were considered.  In \citep{MT}, they studied stochastic proximal gradient algorithms to solve the three terms problem as in \citep{BZ, YMS}. They presented a mini-batch stochastic proximal algorithm for the general stochastic problem and incorporated the variance reduced gradient estimator  into proximal algorithm to solve the finite-sum optimization problem. The differences  and connections between these stochastic algorithms and ours are elaborated in the next subsection.

\subsection{Contributions} \label{contributions}

The main contributions of this paper are elaborated as follows.  

\begin{itemize}
\item {\bf Model.}  We present a new model \eqref{binary-model} for training DNNs with quantized weights. Compared with the  quantization models  in the previous literatures \citep{CBD, YZLOQX}, we minimize the distance between the quantized weights and floating point ones, on training floating point parameters in DNNs at the same time. The generalization ability of floating point parameters can be transferred to quantized parameters in a certain extent, as the loss function is continuous and  is verified in the numerical experiments.

\item {\bf Algorithm.} For solving the problem \eqref{model} with three-block structures, we propose a new stochastic alternating algorithm. Compared with the stochastic PALM (or SPRING) \citep{DTLDS}, the SPRING algorithm solves the problem with $H(x,y)$ being a finite sum  in which the full gradients of $H(x,y)$ are approximated using the variance-reduced gradient estimators.  Our algorithm mainly focuses on the function $G$ having large scale structure, thus we approximate the full gradient of G by unbiased stochastic gradient estimator. Moreover, in our algorithm $G$ could be a more general large-scale structure, not necessarily a finite sum structure.

In the papers \citep{YMS, MT},  their algorithms deal with the three-block problems that do not involve the cross term $H(x,y)$.   In addition, \citep{MT}  use the stochastic gradient estimator with reduced variance, while our algorithm use the unbiased gradient estimator. 

\item {\bf Convergence Analysis.} We establish the convergence analysis for our algorithm \ref{alg:STAM}  on the condition that  the stochastic gradient estimator $\tilde{\nabla} G$ meets the Expected Smoothness (ES) inequality presented in a recent work \citep{KR}.   As pointed out in \citep{KR}, the ES inequality is the weakest among the  current conditions modelling the behaviour of the second moment of stochastic gradient (also see the subsection \ref{subsec-ES} below). Compared with other methods for solving binary neural networks \citep{CBD, YZLOQX},  we show that our algorithm converges to an $\epsilon$-stationary point of problem \eqref{model} and has an $\mathcal{O}(\epsilon^{-4})$ convergence rate. However, in \citep{CBD, YZLOQX}, the convergence rate analysis for their algorithms were not provided. In  \citep{LDXSSG}  the convergence of the BC algorithm was established with loss function being strongly convex.

\item {\bf Experiments.} We apply our algorithm to train VGG-11\citep{SZ}, VGG-16 \citep{SZ} and ResNet-18 \citep{HZRS} DNNs on two standard datasets: CIFAR-10 \citep{K} and CIFAR-100 \citep{K} with relaxed binary weights respectively. The experimental results show the effectiveness of our algorithm. In particular, for CIFAR-10 data set and the DNN with VGG-11 structure, our test accuracy are far better than existing quantization DNN methods. 

\end{itemize}

The rest of this paper is organized as follows. We first review unbiased gradient estimator in Section \ref{sec:convergence}. The convergence rate of the STAM Algorithm \ref{alg:STAM} for  non-convex problems is presented in Section \ref{sec:convergence}. In Section \ref{sec:examples}, we carry out some experiments with the STAM algorithm, and the numerical results show that this algorithm is efficient compared to other algorithms for quantized DNNs. In Section \ref{conclusion}, we give some concluding remarks for our algorithm.

\section{Gradient Estimator and Convergence}\label{sec:convergence}

In this section, we first give the definition of unbiased gradient estimator and some corresponding sampling methods. Then we recall the important Expected Smoothness  (ES) assumption on gradient estimators  proposed by \citep{KR} and prove that this ES assumption is naturally valid in $G$ with a special finite-sum structure.  Finally,  we establish the convergence and convergence rate of Algorithm \ref{alg:STAM} based on the ES assumption.

\subsection{Unbiased Gradient Estimator}\label{UGE01} 
\label{unbiased}
In our theoretical analysis and experiments we always assume that the stochastic gradient estimator $\tilde{\nabla} G(y)$ is unbiased, and its definition is given in the following.    
\begin{definition}\label{ass1} The stochastic gradient estimator $\tilde{\nabla} G(y)$ is unbiased if 
\begin{equation}\label{ass-unbiased}
\mathbb{E} \left[ \tilde{\nabla} G(y) \right] = \nabla G(y). 
\end{equation}
\end{definition}

\begin{remark}
In many applications of deep learning, function $G$ generally has the following finite-sum structure 
\begin{equation}\label{eq500}
G(y) =  \frac{1}{N} \sum_{i=1}^{N} G_i (x), 
\end{equation}
such as the empirical risk in supervised machine learning. In this setting,  the source of stochasticity comes from the way of sampling from the sum, and the unbiased stochastic gradient can be written in the following unified form: 
\begin{equation}\label{stochastic G}
\tilde{\nabla} G(y) = \frac{1}{N}\sum_{i =1}^{N} v_i \nabla G_i (y),
\end{equation}
where  $(v_1, \dots, v_N)$ is a sampling vector drawn from certain distribution $\mathcal{D}$.  The random variable $v_i$ has different forms for different sampling methods. Here we give a representative sampling distribution:
\begin{itemize}
\item {\bf $b$-nice sampling without replacement.} This is a well-known method in deep learning, and we also use this sampling method in our experiments.  We generate a random subset $S \subset \{ 1, 2, \dots, N \}$ by uniformly choosing from all subsets of size $b$, where $b \in [1, N]$ is an integer.  Then we define $ v_i = \frac{1_{i \in S}}{p_i}$ with $1_{i \in S} = 1$ for $i \in S$ and $0$ otherwise, and $p_i = \frac{b}{N}$ for all $i$.
\end{itemize}
There are many other sampling methods that make stochastic gradients unbiased, such as sampling with replacement, independent sampling without replacement. We refer to \citep{KR} for more details.   
\end{remark}

\subsection{Second  Moment of Stochastic Gradient}
\label{subsec-ES}

In this subsection, we review the so-called Expected Smoothness (ES) assumption on the second moment of  stochastic gradient proposed recently by \citep{KR}, related to  the work \citep{RT,GRB,GLQSSR} for stochastic gradient descent (SGD) in the convex setting. This (ES) assumption can be applied to non-convex problems and is essential for our convergence analysis.  

Given a function $G$, from now on we use the  notation $G^{\inf}=\inf_{y \in \mathbb{R}^n} G(y)$.   In the following, we give the ES assumption.  
\begin{assumption}[{\bf Expected Smoothness.}]\label{ass3} 
The second moment of the stochastic gradient $\tilde{\nabla} G(y)$ satisfies 
\begin{equation}\label{422}
\mathbb{E} \left[ \| \tilde{\nabla} G(y) \|^2 \right] \leq 2A_0 \left( G(y) - G^{\inf} \right) + B_0 \| \nabla G(y) \|^2 + C_0,
\end{equation}
for some $A_0, B_0, C_0 \geq 0$ and for all $y \in \mathbb{R}^n$. 
\end{assumption}

\begin{remark}\label{rem1}

When analyzing the convergence of stochastic gradient descent (SGD), various assumptions on the second moment of stochastic gradients have been proposed in many literatures. \citep{KR} have shown that  ES is the weakest,  and hence the most general, among all these assumptions, including such as bounded variance (BV) \citep{GL}, maximal strong growth (M-SG) \citep{T,SR},  expected strong growth (E-SG)  \citep{VBS}, relaxed growth (RG) \citep{BCNo}, and gradient confusion (GC) \citep{SDXHG}. We refer to Section 3 in \citep{KR} for more details. 
\end{remark}

In \citep{KR}, the author showed that when $G$ has a finite-sum structure \eqref{eq500}, the ES inequality automatically holds under mild conditions on the functions $G_i$ and the sampling vectors $v_i$.  In our convergence analysis, we require that $\tilde{\nabla} G(y) + \nabla_y H(x,y)$ satisfies the ES assumption. Next, we would also prove that the ES assumption of $\tilde{\nabla} G(y) + \nabla_y H(x,y)$ is automatically satisfied under natural conditions on the functions $G_i, H$  and the sampling vectors $v_i$ when $G$ has a finite-sum structure \eqref{eq500}. 

\begin{lemma}\label{es-lemma}
Let $G$ be a function with the form \eqref{eq500} and let each $G_i$ be bounded from below by $G_i^{inf}$ and be $L_1^i$ smooth. Let the function $H$ satisfy  $(a2), (a3)$ in Assumption \ref{ass2}.   Suppose that $\mathbb{E} [v_i^2]$ is finite for all $i$ and that $G$ and $H$ is bounded from below by  $G^{\inf}$ and $H^{\inf}$,  respectively.  Define $\Delta^{inf} = \frac{1}{N} \sum_{i=1}^{N} \bigg( (G+H)^{inf} - G_i^{inf} - H^{inf} \bigg)$.   Then $\Delta^{inf} \geq 0$ and the following ES inequality 
\begin{equation}
\mathbb{E} \left[ \| \tilde{\nabla} G(y) + \nabla_y H(x, y) \|^2 \right] \leq 2A \left( G(y) + H(x, y) - (G+H)^{inf} \right) + C 
\end{equation}
holds, where $A =\left( 4 \max_i L_1^{i} \mathbb{E} [v_i^2] +L_3^* \right)/ 2$ and $C = 2A \Delta^{inf}$.
\end{lemma}

See the proof of Lemma \ref{es-lemma} in Appendix \ref{app1}.

\subsection{Convergence Analysis}
\label{convergence}
We now establish the convergence rate of the STAM Algorithm \ref{alg:STAM}.  In analyzing STAM Algorithm \ref{alg:STAM},  we need the following mild assumptions on the non-convex  functions $G$ and $H$.  
\vspace{0.1cm}
\begin{assumption}\label{ass2} Functions $G$ and $H$ satisfy  
\vspace{0.1cm}
\begin{itemize}
\item [(a1)] $G$ is bounded from below by $G^{inf}$, and $G$ has a Lipschitz continuous gradient, i.e, there exists a constant $L_1 > 0$ such that 
\begin{equation}\label{G-lip}
\| \nabla G(y_1) - \nabla G(y_2) \| \leq L_1 \| y_1 - y_2 \|, ~~~\forall y_1, y_2 \in \mathbb{R}^{n}.
\end{equation}

\item [(a2)] There exist $L_2(y), L_4(y) > 0$ such that
\begin{equation}\label{H-xlip}
\| \nabla_x H(x_1, y) - \nabla_x H(x_2, y) \| \leq L_2(y) \| x_1 - x_2 \|, ~~~\forall x_1, x_2 \in \mathbb{R}^{n},
\end{equation}
\begin{equation}\label{Hx-ylip}
\| \nabla_x H(x, y_1) - \nabla_x H(x, y_2) \| \leq L_4(x) \| y_1 - y_2 \|, ~~~\forall y_1, y_2 \in \mathbb{R}^{n}
\end{equation}
and there exist $L_3(x), L_5(x) > 0$ such that
\begin{equation}\label{H-ylip}
\| \nabla_y H(x, y_1) - \nabla_y H(x, y_2) \| \leq L_3(x) \| y_1 - y_2 \|, ~~~\forall y_1, y_2 \in \mathbb{R}^{n},
\end{equation}
\begin{equation}\label{Hy-ylip}
\| \nabla_y H(x_1, y) - \nabla_y H(x_2, y) \| \leq L_5(y) \| x_1 - x_2 \|, ~~~\forall x_1, x_2 \in \mathbb{R}^{n},
\end{equation}
\item[(a3)] There exist $L_2^{*}, L_3^{*}, L_4^{*} \geq 0$ and $L > 0$ such that
 \begin{equation}\label{lip-const}
 \begin{aligned}
& \sup_{y \in \mathbb{R}^n} L_2(y) \leq L_2^{*}, \quad \sup_{x \in \mathbb{R}^n} L_3(x) \leq L_3^{*}, \\
& \sup_{x \in \mathbb{R}^n} L_4(x) \leq L_4^{*}, \quad \sup_{y \in \mathbb{R}^n} L_5(y) \leq L_5^{*},  \quad \textmd{and} \quad L_3^{*} + L_1 \leq L.
\end{aligned}
 \end{equation}
\end{itemize}
\end{assumption}

In addition, let $l \in \mathbb{R}$ be such that $H(\cdot, y) + \frac{l}{2}\| \cdot \|^2 $ is convex for all $y \in \mathbb{R}^n$. Note that such $l$ always exists by \eqref{H-xlip} and \eqref{lip-const}.   Particularly, one can always take $l=L_2^*$ .

If the stochastic gradient $\tilde{\nabla} G$ satisfies the ES inequality \eqref{422}, then the following lemma shows that $\tilde{\nabla} G(y) + \nabla_y H(x,y)$ also satisfies the corresponding ES assumption. This result is essential for analyzing the convergence of STAM Algorithm \ref{alg:STAM}. 
\begin{lemma}\label{ES-GH}
Suppose that the stochastic gradient $\tilde{\nabla} G$ satisfies Assumption \ref{ass3} and $G, H$  satisfy Assumption \ref{ass2}. Then we have \begin{equation}\label{gh-es}
\mathbb{E}\left[\|\tilde{\nabla} G(y) + \nabla_y H(x,y) \|^2 \right] \leq 2A \left[ G(y) + H(x, y) - (G+H)^{inf} \right] + C, 
\end{equation}
where $A = \max (2A_0 + 2B_0 L_1, 2 L_3^*)$ and $C = 2A \left[ (G+H)^{inf} - G^{inf} - H^{inf} \right] + 2C_0$.
\end{lemma}

We refer the proof of Lemma \ref{ES-GH} to Appendix \ref{app1}.  For $G$ being a finite-sum as in \eqref{eq500}, from Lemma \ref{es-lemma} we can see that $\tilde{\nabla} G(y) + \nabla_y H(x,y)$ automatically satisfied the ES assumption under mild conditions on $G_i$ and $H$.  

The following lemma plays an important role in establishing the convergence of Algorithm \ref{alg:STAM} when applied to solving the non-convex problem \eqref{model}.

\begin{lemma}\label{conv-lemma} 
Let $\{ (y^t, x^t, u^t, z^t) \}$ be a sequence generated from the STAM Algorithm \ref{alg:STAM}. Suppose that the stochastic gradient $\tilde{\nabla} G$ satisfies Assumption \ref{ass3} and $G, H$  satisfy Assumption \ref{ass2}. Suppose that the parameters $\beta >0$  and $\gamma > 0$ are chosen such that 
\begin{equation}\label{k1}
\mathcal{K}_1 :=  \frac{ 1- (10(l + L_2^*)+4)\gamma -5 (L_2^*)^2 \gamma^2}{4 \gamma}  > 0. 
\end{equation}
Then
\begin{equation}\label{ineq1}
\begin{aligned}
&\sum_{t= 0}^{T-1} \omega_t \eta_t + \frac{\beta \mathcal{K}_1}{\mathcal{K}_2} \sum_{t=0}^{T-1} \omega_t \mathbb{E}\left(  \frac{\| z^t - z^{t-1}\|}{\gamma} \right)^2 \leq \beta \omega_{-1} \delta_0 + \beta \omega_{-1} \mathcal{M}_{0}^{\prime}  + \frac{(L+M)C}{2\beta} \sum_{t=0}^{T-1} \omega_t, \\
\end{aligned}
\end{equation}
where 
\[
\eta_t = \mathbb{E} \| \nabla G(y^t) + \nabla_y H(x^t, y^t) \|^2, ~~\omega_t = \frac{\omega_{-1}}{\left(1 + \frac{(L+M)A}{\beta^2} \right)^{t+1}}~~ \textmd{for} ~~\omega_{-1} >0~~arbitrary, 
\]
$\mathcal{M}_{0}^{\prime} = \mathbb{E} \left[ \mathcal{M}_{0} - \inf_{t\geq 0} \mathcal{M}_t \right]$ with 
$$
\mathcal{M}_t := \mathcal{M} (x^t, u^t, z^t) = F(u^{t}) + \frac{1}{2\gamma} \| 2x^{t} - u^{t} - z^{t} \|^2 - \frac{1}{2\gamma}\| x^{t} - z^{t} \|^2 - \frac{1}{\gamma} \| u^{t} - x^{t} \|^2,  
$$
 $\delta_t = \mathbb{E} \left[ G(y^t) + H(x^t, y^t) - (G+H)^{inf} \right]$,  $\mathcal{K}_2 := \frac{5(1 + \gamma L_2^*)^2}{4\gamma^2}$ and $M=2  \mathcal{K}_3$ with 
$$
 \mathcal{K}_3 := L_4^* ( 1 + 5\gamma L_4^* ) + \frac{5 \mathcal{K}_1 (L_4^*)^2}{\mathcal{K}_2}. 
$$
\end{lemma} 

\begin{remark}\label{Rem=09836}
Notice that $\lim_{\gamma \to 0^+}\mathcal{K}_1 = +\infty$. Therefore, for any given $l \in \mathbb{R}$ and $L_2^*\geq 0$, the condition \eqref{k1} will be satisfied when $\gamma>0$ is  sufficiently small. Moreover, using the quadratic formula we easily obtain the following computable threshold 
\[
0 < \gamma < \frac{\sqrt{ ( 10 L_2^* + 10 l +4 )^2 + 20 (L_2^*)^2} - (10 L_2^* + 10 l +4)}{10 L_2^*}
\]
such that \eqref{k1} holds. 
\end{remark}

See the proof of Lemma \ref{conv-lemma}  in Appendix \ref{app1}.  Lemma \ref{conv-lemma}  provides a bound on a weighted sum of stochastic gradients about $y$ and deterministic gradients about $x$.  The similar idea of weighting different iterations has also been used in the  analysis of stochastic gradient descent in the convex setting \citep{RSS, SZh, SU} and in the non-convex setting \citep{KR}.

\begin{theorem}\label{conv}
Let $\{ (y^t, x^t, u^t, z^t) \}$ be a sequence generated from STAM Algorithm \ref{alg:STAM}. Suppose that the stochastic gradient $\tilde{\nabla} G$ satisfies Assumption \ref{ass3} and $G, H$  satisfy Assumption \ref{ass2}.  Suppose that the parameters $\gamma, \beta >0$  are chosen such that \eqref{k1}, $\left[(L_2^*)^2 + (L_5^*)^2 \right] \gamma^2 \leq 1$ and $\beta \mathcal{K}_1/ \mathcal{K}_2 \geq 2$ hold.  Then we have the following estimate: 
\begin{equation}\label{con-inequality}
\begin{aligned}
&\min_{0 \leq t \leq T-1} \bigg[ \mathbb{E} \| \nabla G(y^t) + \nabla_y H(u^t, y^t) \|^2 + \mathbb{E}~\textmd{dist}^2 \big( 0,  \nabla_x H(u^t, y^t) + \partial F(u^t) \big) \bigg] \\
&\leq  2\frac{\left( 1+ \frac{(L+M)A}{\beta^2} \right)^{T}}{T} \beta ( \delta_0 + \mathcal{M}_{0}^{\prime} ) + \frac{(L+M)C}{\beta}, 
\end{aligned}
\end{equation}
where the constants $\delta_0$,  $ \mathcal{M}_{0}^{\prime}$ and $M$ are defined as in Lemma \ref{conv-lemma}. 
\end{theorem}

 Although the bound of Theorem \ref{conv} shows possible exponential growth, by carefully controlling the parameters we could obtain an $\epsilon$-stationary point  with the optimal $\mathcal{O}(\epsilon^{-4})$ rate. More precisely, we have the following convergence rate when using our STAM Algorithm \ref{alg:STAM} to find an $\epsilon$-stationary point of the non-convex optimization problem \eqref{model}.

\begin{theorem}[{\bf Convergence rate.}]\label{conv-rate}
Let $\{ (y^t, x^t, u^t, z^t) \}$ be a sequence generated from STAM Algorithm \ref{alg:STAM}. Suppose that the stochastic gradient $\tilde{\nabla} G$ satisfies Assumption  \ref{ass3} and $G, H$  satisfy Assumption \ref{ass2}.  Suppose that the parameter $\gamma >0$ is chosen such that \eqref{k1} and $\left[(L_2^*)^2 + (L_5^*)^2 \right] \gamma^2 \leq 1$ hold.   Given $\epsilon > 0$, choose the parameter $\beta  > \min \Big\{ \sqrt{(L+M)AT},$\\
~$\frac{2\mathcal{K}_2}{\mathcal{K}_1}, ~\frac{2(L+M)C}{\epsilon^2}\Big\}$. If 
\begin{equation}
T \geq \frac{12(L+M)(\delta_0 + M_0^{\prime})}{\epsilon^2} \max \Big\{ \frac{2C}{\epsilon^2}, \frac{12(\delta_0+ M_0^{\prime})A}{\epsilon^2}, \frac{2\mathcal{K}_2}{\mathcal{K}_1(L+M)} \Big\} = \mathcal{O}(\epsilon^{-4}), 
\end{equation}
then there exists $0\leq t_0 \leq T-1$ such that 
\begin{equation}
\begin{aligned}
& \mathbb{E} \| \nabla G(y^{t_0})) + \nabla_y H(u^{t_0}, y^{t_0}) \| \leq \epsilon~~\textmd{and}~~ \mathbb{E}~\textmd{dist}\big(0, \nabla_x H(u^{t_0}, y^{t_0}) + \partial F(u^{t_0}) \big) \leq \epsilon.
\end{aligned}
\end{equation}
\end{theorem}

For the clarity of the presentation, we refer the  proof of Theorem \ref{conv} and Theorem \ref{conv-rate} to  Appendix \ref{P=MT0}.

\section{Experiments}
\label{sec:examples} 
In this section, we give the numerical experiments of training DNNs with relaxed binary weights using  Algorithm \ref{alg:STAM}. Our experiments are mainly performed on three different network structures, VGG-11 \citep{SZ}, VGG-16 \citep{SZ}, and ResNet-18 \citep{HZRS}.  These DNNs are trained  using two different data sets CIFAR-10 \citep{K} and CIFAR-100 \citep{K} respectively.  All experiments are implemented in  PyTorch platform with Python 3.6. The experiments are run on a remote desktop with a Tesla V100 GPU and 32GB memory. 

\begin{figure}[!ht]
	\centering
	\subfloat[{\tt {\bf CIFAR-10}}]{ \includegraphics[width=0.40\linewidth]{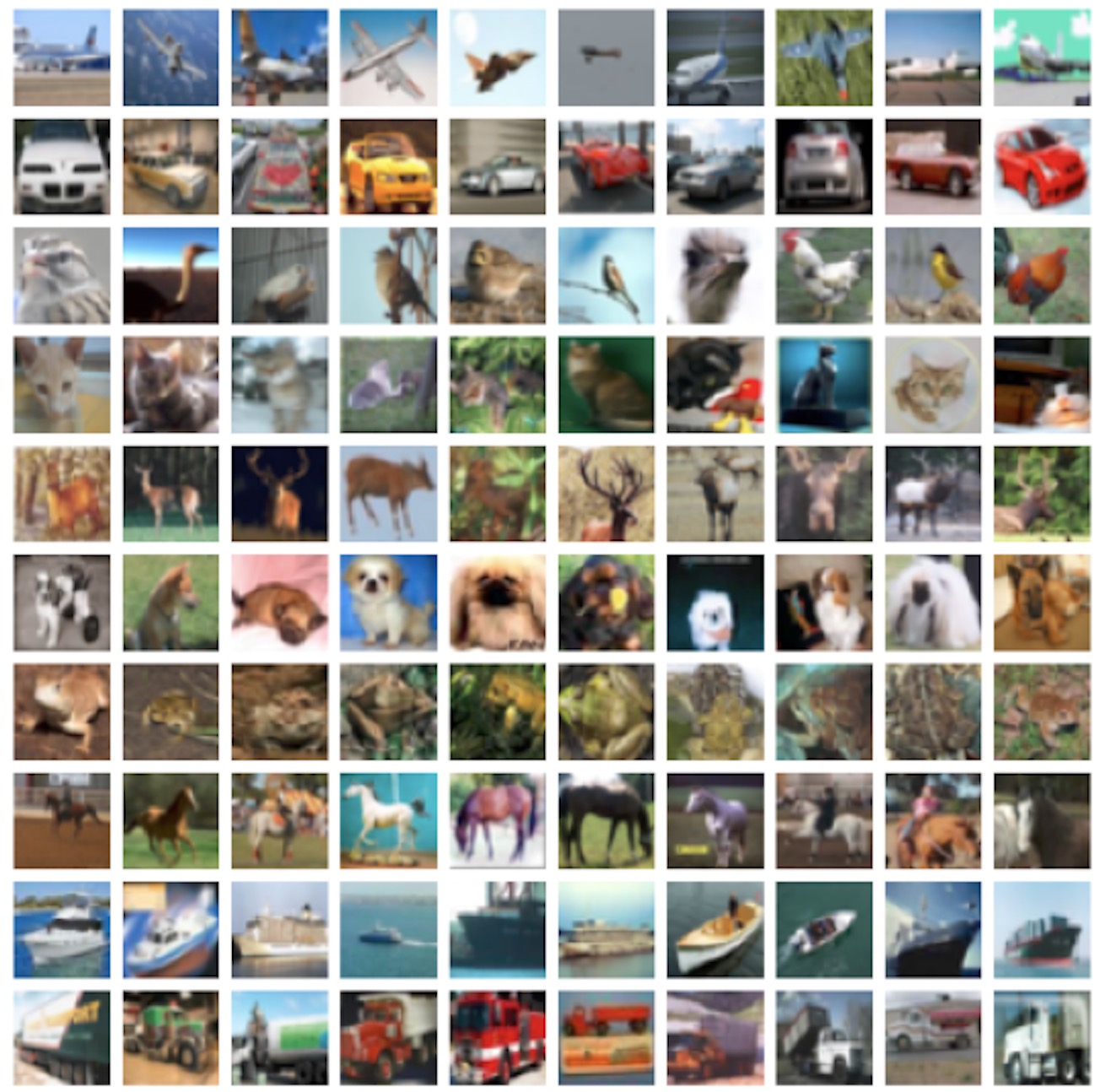} } 
	\hspace{0.3cm}		
	\subfloat[{\tt {\bf CIFAR-100 }}]{ \includegraphics[width=0.40\linewidth]{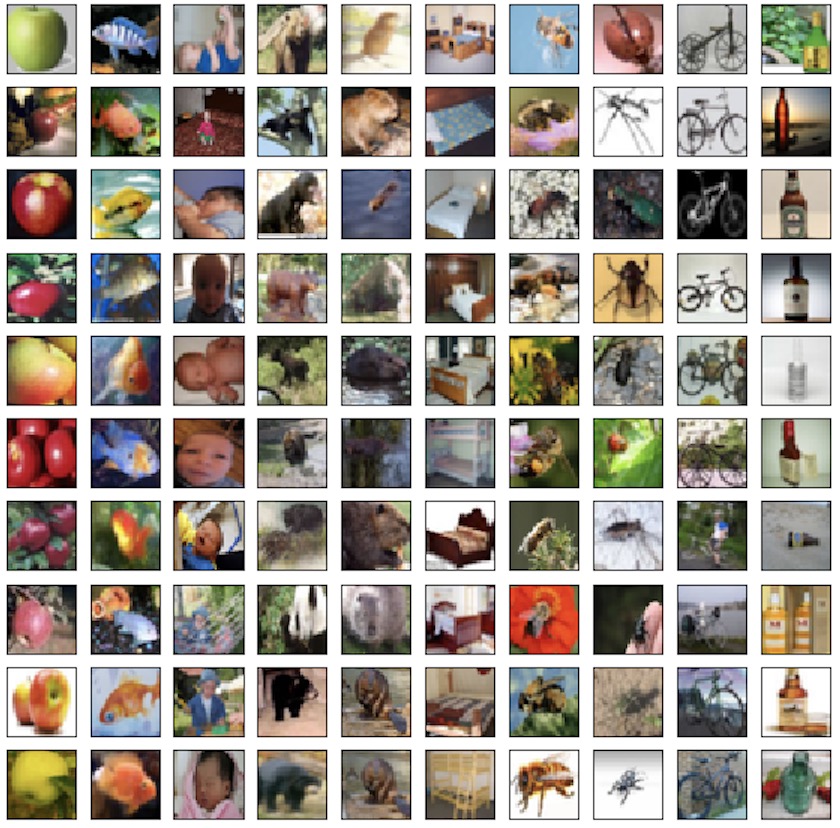} } 		

	\caption{ Sampled images from {\bf CIFAR-10} and {\bf CIFAR-100} training datasets.}
    \label{datasets}
\end{figure}

\subsection{Algorithms}
\label{3.1}
In all experiments, we compare our Algorithm \ref{alg:STAM} with the BinaryConnect(BC) \citep{CBD}, BinaryRelax(BR) \citep{YZLOQX} and proximal stochastic gradient descent (PSGD) \citep{CP, RVV} algorithms. The BC algorithm has become one of the most important algorithms for training quantized DNNs (such as Xnor-net). The BinaryRelax (BR) algorithm is a relaxed two-stage algorithm proposed in \citep{YZLOQX} and has been shown to be effective for training quantized DNNs. BC \citep{CBD} and PSGD \citep{CP, RVV}  trained DNNs by minimizing the following problem:
\begin{equation}\label{eq02}
\min_{\tilde{W}} L(\tilde{W}) + \mathcal{I}_{Q} (\tilde{W}),
\end{equation}
where the $L(\cdot)$ is the loss function of DNN and $\mathcal{I}_{Q}$ is the indicator function of the quantitative set $Q$ defined as 
\begin{equation}\label{eq01}
Q:=\{ {W}_{ij}  ~|~  | W_{i1} | = | W_{i2} | = \dots =| W_{in} |=s_i, W_{ij} \in s_i \times \{ -1, +1 \}, s_i \in \mathbb{R}^{+} \}
\end{equation}
with $i$ being the $i$-th layer of DNN and $n$ being the number of neurons at $i$-th layer. The magnitude $s_i$ in \eqref{eq01} has been calculated precisely by \citep{RORF}. To keep the paper self-contained, here we recall the calculation details of $s_i$.

The projection of floating-point weights onto quantized set $Q$ is to solve the following optimization problem:
\begin{equation}\label{0411}
\tilde{W}^* \in \arg\min_{\tilde{W} \in Q} \| \tilde{W} - U \|^2 := \mbox{Proj}_{Q} (U).
\end{equation}
According to the definition of $Q$, the projection problem \eqref{0411} can be reformulated as 
\begin{equation}\label{0412}
(s_i^*, Z^*) = \arg\min_{s_i, Z} \| s_i \cdot Z - U_i \|^2, ~~\textmd{s.t.}~~ Z \in \{-1, +1\}^n,
\end{equation}
where $U_i$ denotes the weights of the $i$-th layer. It has been shown in \citep{RORF} that the closed (exact) solution of \eqref{0412} can be obtained as follows:
\begin{equation}
s_i^* = \frac{\|vec(U_i)\|_1}{n}, \quad Z_{i,j} = \begin{cases} 1 \quad  &\textmd{if}~U_{ij} \geq 0, \\ -1  \quad &\textmd{otherwise}. \end{cases}
\end{equation}

When PSGD and BC algorithms are applied to problem \eqref{eq02}, the specific form is respectively given as  
\begin{equation}\label{alg:PSGD}
(\textmd{PSGD})~\begin{cases}
U^{t+1} \in U^t - \gamma \tilde{\nabla} L(U^t) ,\\
\tilde{W}^{t+1} \in  \mbox{Proj}_{Q}(U^{t+1}) 
\end{cases}
\end{equation}
and 
\begin{equation}\label{alg:BC}
(\textmd{BC})~\begin{cases}
U^{t+1} \in U^t - \gamma \tilde{\nabla} L(\tilde{W}^t) ,\\
\tilde{W}^{t+1} \in  \mbox{Proj}_{Q}(U^{t+1}).
\end{cases}
\end{equation}

Notice that the BinaryRelax (BR) \citep{YZLOQX}  is a two-stage algorithm. In the first stage,  BR algorithm minimizes the following problem
\begin{equation}\label{BR}
\min_{\tilde{W}} \frac{\lambda}{2} \textmd{dist} (\tilde{W}, Q)^2 + L(\tilde{W}),
\end{equation}
where $\lambda$ is the regularization parameter, $L(\cdot)$ is the loss function of DNN. In the second stage, BR algorithm solves the  problem \eqref{eq02} as BC and PSGD. The specific BR algorithm is given as
\begin{equation}\label{alg:BR}
(\textmd{BR})\begin{cases}
\tilde{W}^{t+1}  = \begin{cases}  \frac{\lambda_t Proj_{Q}(U^{t+1}) + U^{t+1} }{\lambda_t + 1} (\lambda_t = \rho \lambda_t, \rho >1)  \qquad &\textmd{if } t < T, \\ 
\\
Proj_{Q}(U^{t+1}) \qquad \qquad  \qquad \qquad \qquad \qquad  &\textmd{if } t \geq  T, \end{cases} \\
\\
U^{t+1} \in U^t - \gamma^t \tilde{\nabla} L(\tilde{W}^t).\\
\end{cases}
\end{equation}

When our  Algorithm \ref{alg:STAM} is applied  to train DNN with quantized weights,  we solve the following model 
\begin{equation}\label{stam}
\min_{W, \tilde{W}} \frac{\lambda}{2} \| W - \tilde{W} \|_{F}^2 + L_{W} (x, y) + \mathcal{I}_{\mathcal{Q}}(\tilde{W}), 
\end{equation}
where $L(\cdot)$ is the loss function of DNN and $\mathcal{I}_{Q}$ is the indicator function of the quantitative set as in \eqref{eq01}. The specific algorithm is presented as follows:

\begin{equation}\label{alg:stamm}
(\textmd{STAM})\begin{cases}
W^{t+1} \in \frac{(\beta - \lambda) W^t +\lambda \tilde{W}^t - \tilde{\nabla} L(W^T)}{\beta},\\
\\
U^{t+1} \in  \frac{\gamma \lambda W^{t+1} + X^t}{\gamma \lambda + 1},\\
\\
V^{t+1} \in  \mbox{Proj}_{Q}(2U^{t+1} - X^t),\\
\\
X^{t+1} = X^{t} + (V^{t+1} - U^{t+1}).
\end{cases}
\end{equation}

\subsection{CIFAR-10 dataset}
\label{3.2}
In this subsection, we train DNNs on the CIFAR-10 dataset \citep{K} using the four different algorithms presented in subsection \ref{3.1}. The CIFAR-10 dataset consists of 10 categories of $32 \times 32$ color images containing a total of $60,000$ images, with each category containing $6,000$ images (see Figure \ref{datasets} (a)). A set of $50000$ images are used as the training set and $10000$ as the test set. In this experiment, we compare our algorithm \eqref{alg:stamm} with BC, BR and PSGD algorithms. The parameters selection in all algorithms are given in Table \ref{para}.  In our algorithm, we choose parameters in three ways and denote as Our1,  Our2 and Our3 respectively. The details are given in Table \ref{para}.

In this experiment, we use $1000$ epochs to train DNNs with quantized weights for all algorithms. In the process of training DNN, we set the  batch-size as 128. In BR algorithm, we set the parameter $K=250$ to start the second phase. We finally compare the accuracy of all algorithms on the test and train sets of CIFAR-10. For BC and BR algorithms, we set their parameters as described in their papers \citep{CBD} and \citep{YZLOQX} respectively. 

We show the train and test accuracy in Figure \ref{acc_cifar}, and give the best test accuracy of all methods in Table \ref{cifar-res}. From Figure \ref{acc_cifar}, we can see that Our1 has similar behavior with BC and BR, all of which can achieve relatively high test accuracy with a small number of epochs. Furthermore, from Table \ref{para} we see that the parameters  $\lambda$ and $\gamma$ in Our1 are both very small, hence Our1 gets stuck  when reaching a local optimum.  Even so,  the test accuracy of Our1 is comparable to BR and better than BC. 

Thus we set the parameters $\gamma$ and  $\lambda$ in Our2 are relatively large when the number of epoch is small,  and then decrease $\gamma$ to ensure convergence. This treatment is similar to the gradual decay of learning rate in BC and BR algorithms. In Our3, we combine the above two parameters  in Our1 and Our2  so that the algorithm can  converge to a better local minimum. From Table \ref{cifar-res} we can  see  that both Our2 and Our3 have relatively good test accuracy especially for VGG-11 DNN, and Our3 has the best test accuracy among all these algorithms.  
\begin{table}[htbp]\footnotesize
\centering
\resizebox{\textwidth}{!}{
\begin{tabular}{|c|c|c|c|c|c|c|c|c|}
\hline
\multicolumn{1}{|c|}{Algorithm}&\multicolumn{5}{|c|}{ parameters }\\
\hline
 &$\gamma$ & $\lambda$ & $\rho$ & $\beta$ & weight decay  \\\cline{1-6}
PSGD& \makecell[c]{if ($0< epoch < 240$):\\ lr = 5e-4\\ else ($240\leq  epoch \leq 360$): \\ $epoch/20==0$: lr=lr*0.1}  &  $\times$  & $\times$   & $\times$  & 1e-7  \\\cline{1-6}
BC~&  \makecell[c]{if ($0< epoch < 240$):\\ lr = 5e-4\\ else ($240\leq epoch \leq 360$): \\ $epoch/20==0$: lr=lr*0.1}   & $\times$   &$\times$   &$\times$   & 1e-7  \\\cline{1-6}
 BR&  \makecell[c]{if ($0< epoch < 240$):\\ lr = 5e-4\\ else ($240\leq epoch \leq 360$): \\ $epoch/20==0$: lr=lr*0.1} & 1 & 1.02  &$\times$   & 1e-7  \\\cline{1-6}
Our1& 8 & 0.5& $\times$   &\makecell[c]{if epoch $\leq$ 80: 1e3\\ else: \\
max(0.9999$* \beta$, 920)} & $\times$   \\\cline{1-6}
Our2&\makecell[c]{if epoch $\leq$ 80: \\ $\gamma= 15 $\\ else: \\ $\gamma = \max(0.995*\gamma, 0.3)$}  & 4 & $\times$  & \makecell[c]{if epoch $\leq$ 80: 1e3\\ else:\\
 max(0.9999$* \beta$, 920)} &$\times$  \\\cline{1-6}
Our3 & \makecell[c]{if epoch $\leq$ 80: \\ $\gamma = 8$ \\ else: \\ $\gamma= 15$ (epoch=81),\\ $\gamma = \max(0.995*\gamma, 0.3)$ (epoch $\geq$ 81)}
&\makecell[c]{if epoch $\leq$ 80: \\ $\lambda= 0.5 $\\ else: $\lambda =4$}  
&$\times$   & \makecell[c]{if epoch $\leq$ 80: 1e3\\ else:\\ max(0.9999$* \beta$, 920)} & $\times$  \\\cline{1-6}
 
\end{tabular}
}
\caption{The selection of parameters for different methods.}\label{para}
\end{table}

\begin{figure}[!ht]
	\centering
	\subfloat[{\tt {\bf VGG-11}}]{ \includegraphics[width=0.45\linewidth]{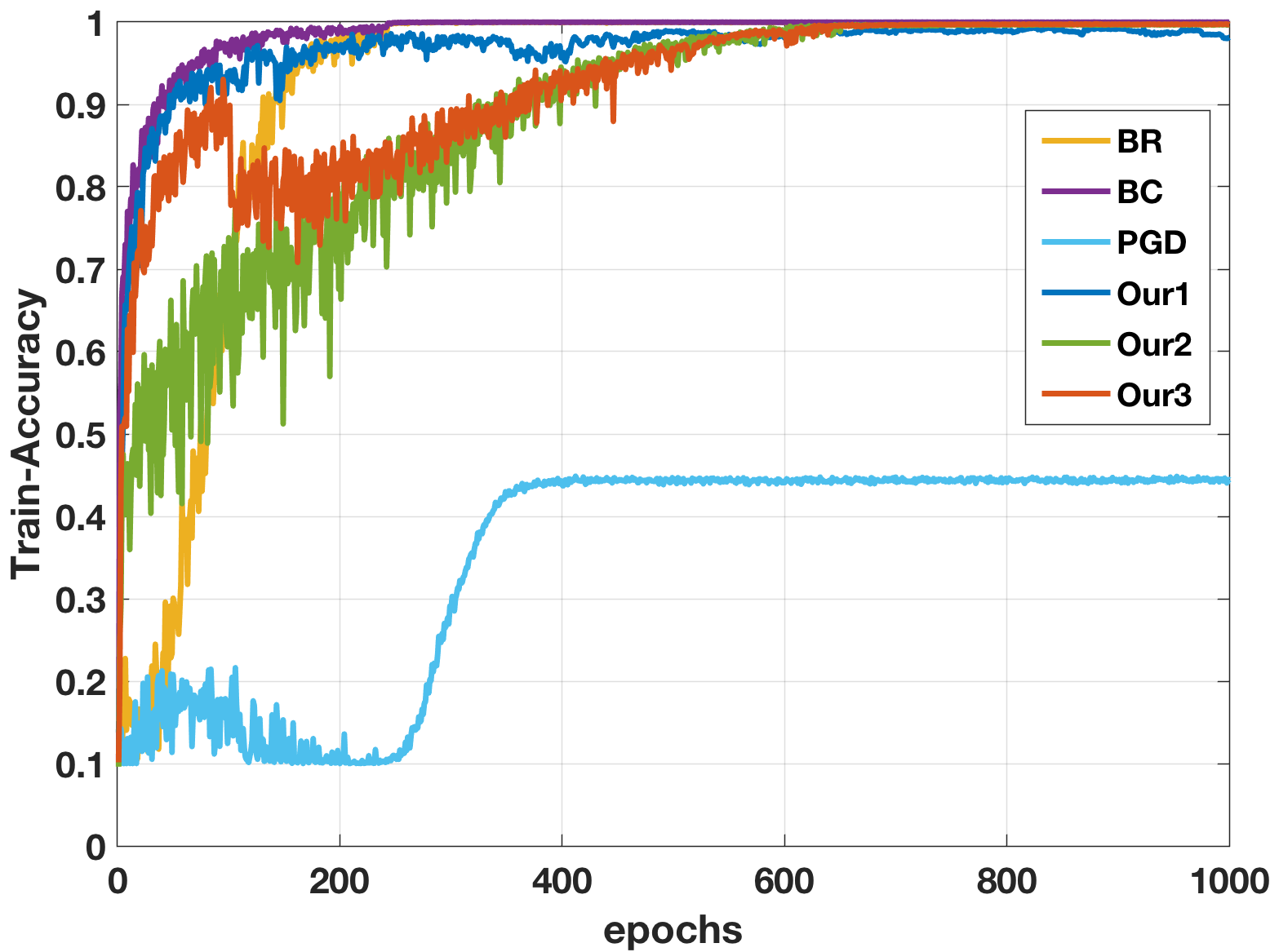} } 
	\hspace{0.1cm}		
	\subfloat[{\tt {\bf VGG-11 }}]{ \includegraphics[width=0.45\linewidth]{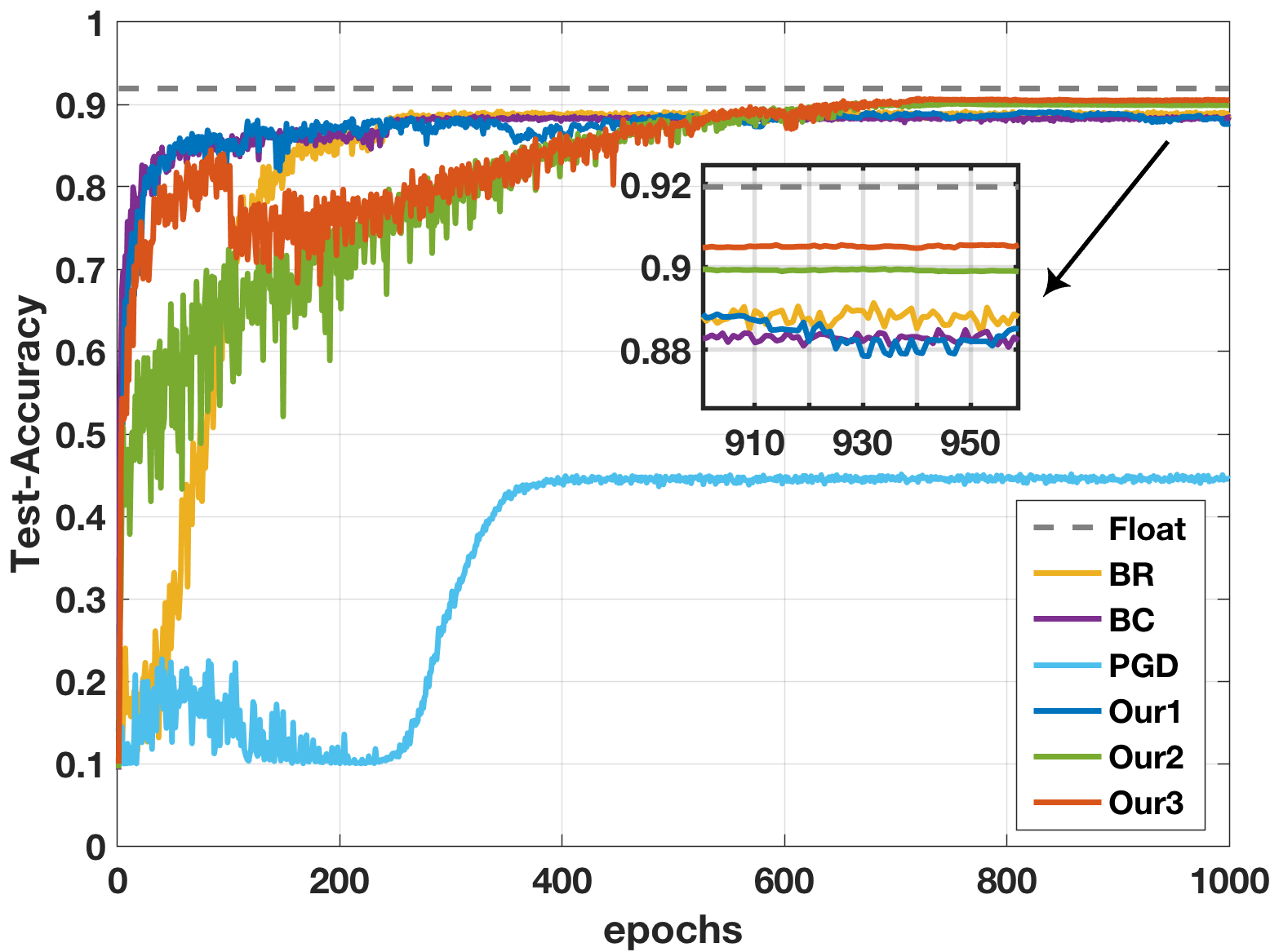} } 	\\[-2mm]	
	\hspace{0.1cm}	
	\subfloat[{\tt {\bf VGG-16 }}]{ \includegraphics[width=0.45\linewidth]{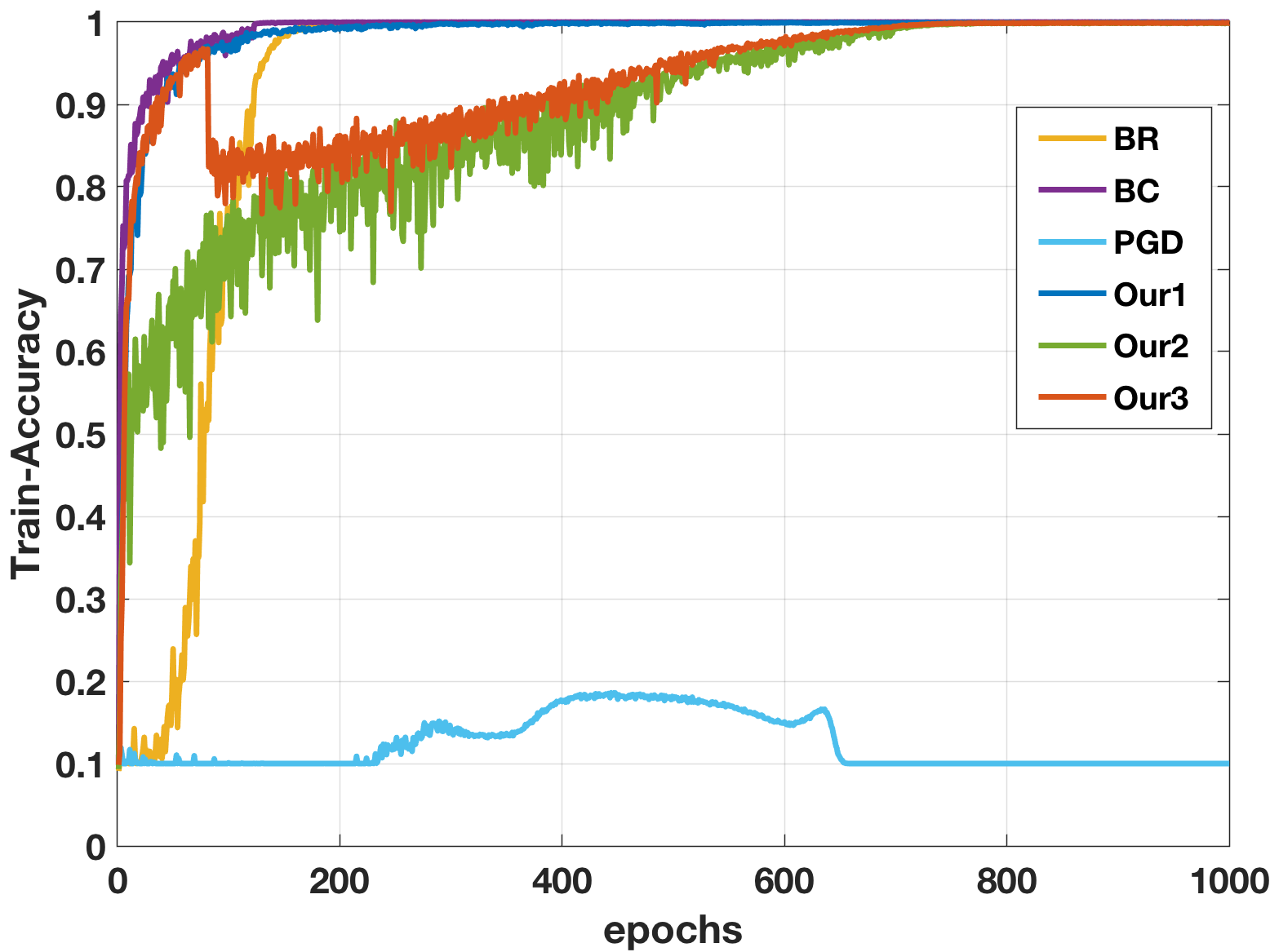} } 	
         \hspace{0.1cm}	
	\subfloat[{\tt {\bf VGG-16 }}]{ \includegraphics[width=0.45\linewidth]{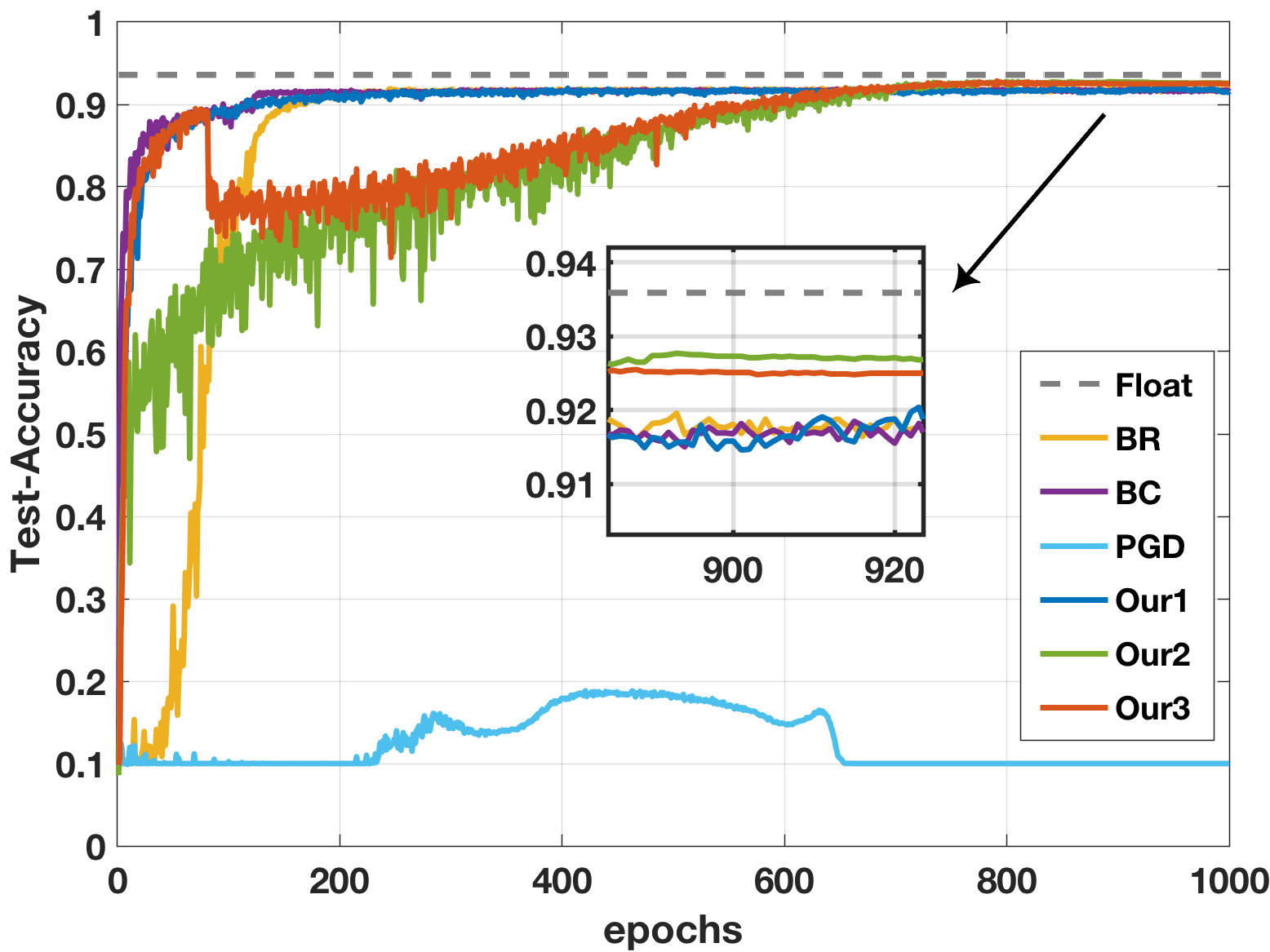} } \\[-2mm]	
		\hspace{0.1cm}	
	\subfloat[{\tt {\bf ResNet-18 }}]{ \includegraphics[width=0.45\linewidth]{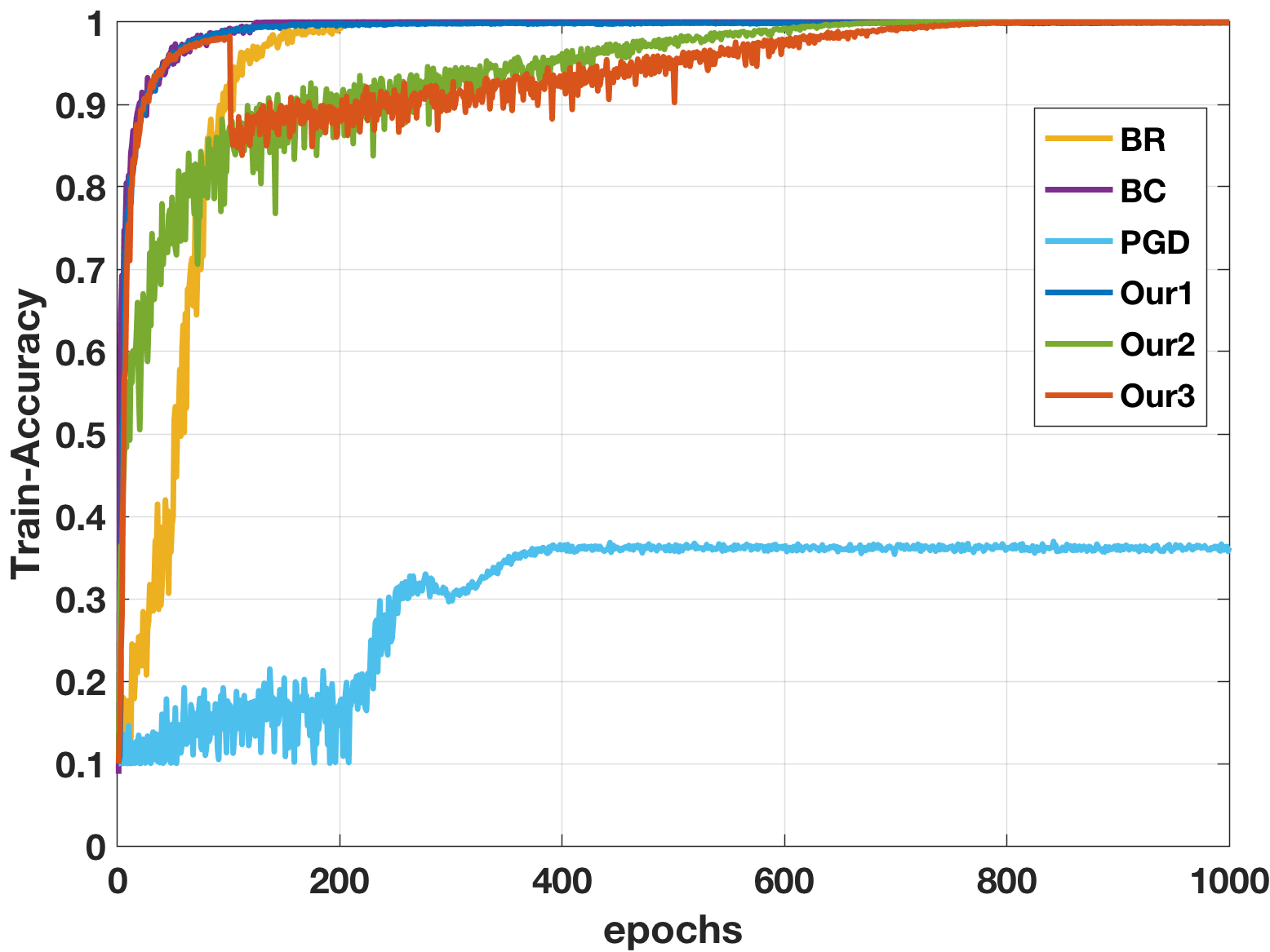} } 	
         \hspace{0.1cm}	
	\subfloat[{\tt {\bf ResNet-18 }}]{ \includegraphics[width=0.45\linewidth]{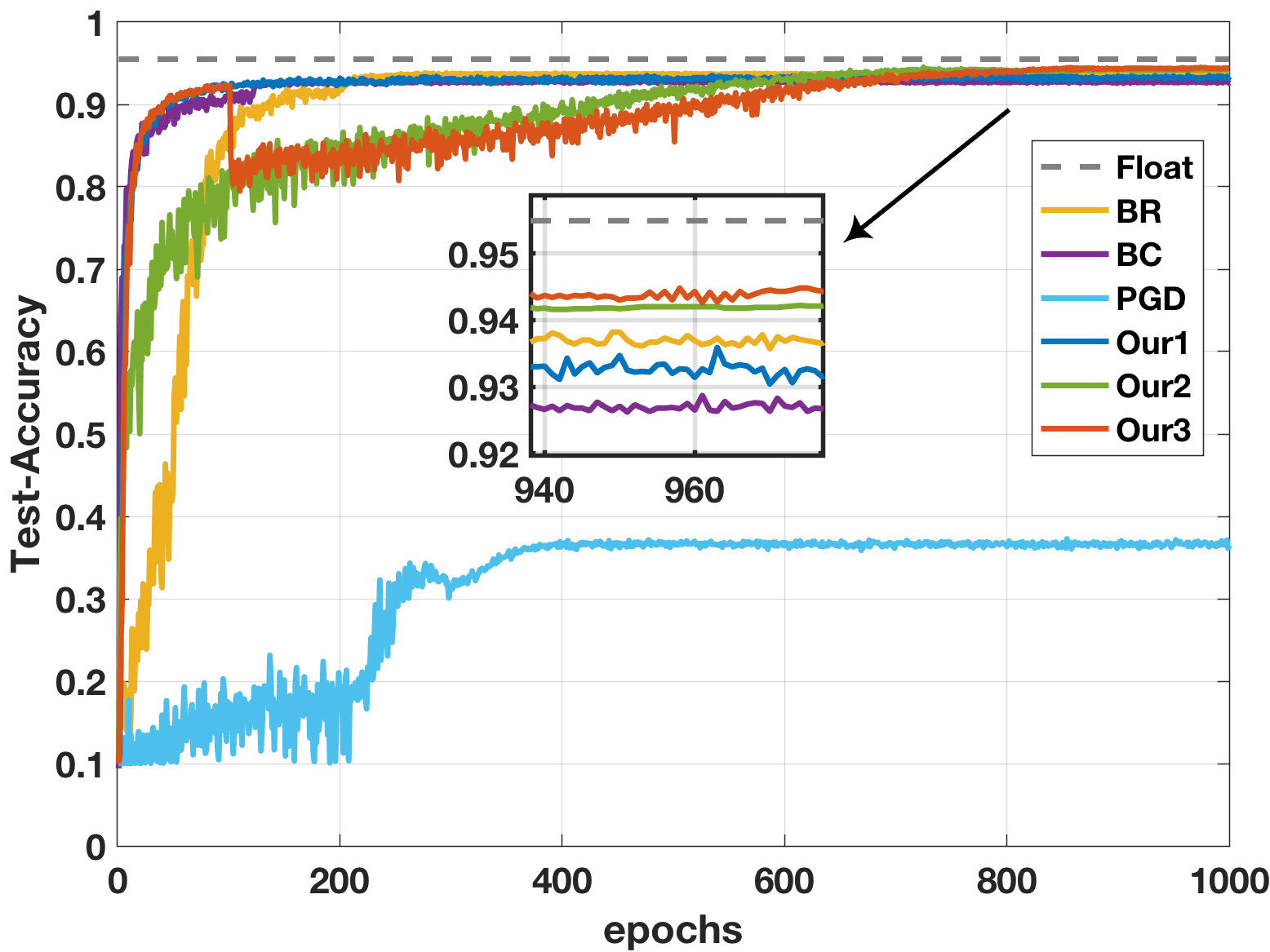} } 
	\caption{ Test accuracy and training accuracy for {\bf CIFAR-10} dataset.}
    \label{acc_cifar}
\end{figure}

\begin{table}[htbp]
\centering
\begin{tabular}{|c|c|c|c|c|c|c|c|}
\hline
DNN &  Float &PSGD &  BC & BR &Our1 & Our2 & Our3 \\ \cline{1-8}
VGG-11  &91.93 & 45.13 & 88.59 & 89.24 & 89.11 & 90.18 & {\bf 90.71} \\ \cline{1-8}
VGG-16 & 93.59  & 18.84  & 91.88 & 92.00 & 92.04 & 92.79 & {\bf 92.85}   \\\cline{1-8}
ResNet-18 & 95.49 & 37.27 & 92.96 & 93.86 & 93. 71 & {\bf 94.52} & 94.49  \\\cline{1-8}
\end{tabular}
\caption{The best test accuracy for different methods.}\label{cifar-res}
\end{table}

\subsection{CIFAR-100 dataset}
\label{sub:3.3}
In this subsection, we perform the test on CIFAR-100 dataset \citep{K}  containing $100$ classes, each of which contains $600$ images (see Figure \ref{datasets} (b)). These $600$ images are divided into $500$ training images and $100$ test images respectively.

\begin{table}[htbp]
\centering
\begin{tabular}{|c|c|c|c|c|c|c|c|}
\hline
DNN &  Float &  BC & BR & Our3 \\ \cline{1-5}
VGG-11  &70.43 & 62.37 & 64.01 & {\bf 64.69} \\ \cline{1-5}
VGG-16 & 73.55   & 69.33 & 70.20  & {\bf 71.26}   \\\cline{1-5}
ResNet-18 & 76.32 & 72.43 & 73.94 & {\bf 74.47}  \\\cline{1-5}
\end{tabular}
\caption{The best test accuracy for different methods.}\label{cifar100-res}
\end{table}

\begin{figure}[htbp]
	\centering
	\subfloat[{\tt {\bf VGG-11}}]{ \includegraphics[width=0.45\linewidth]{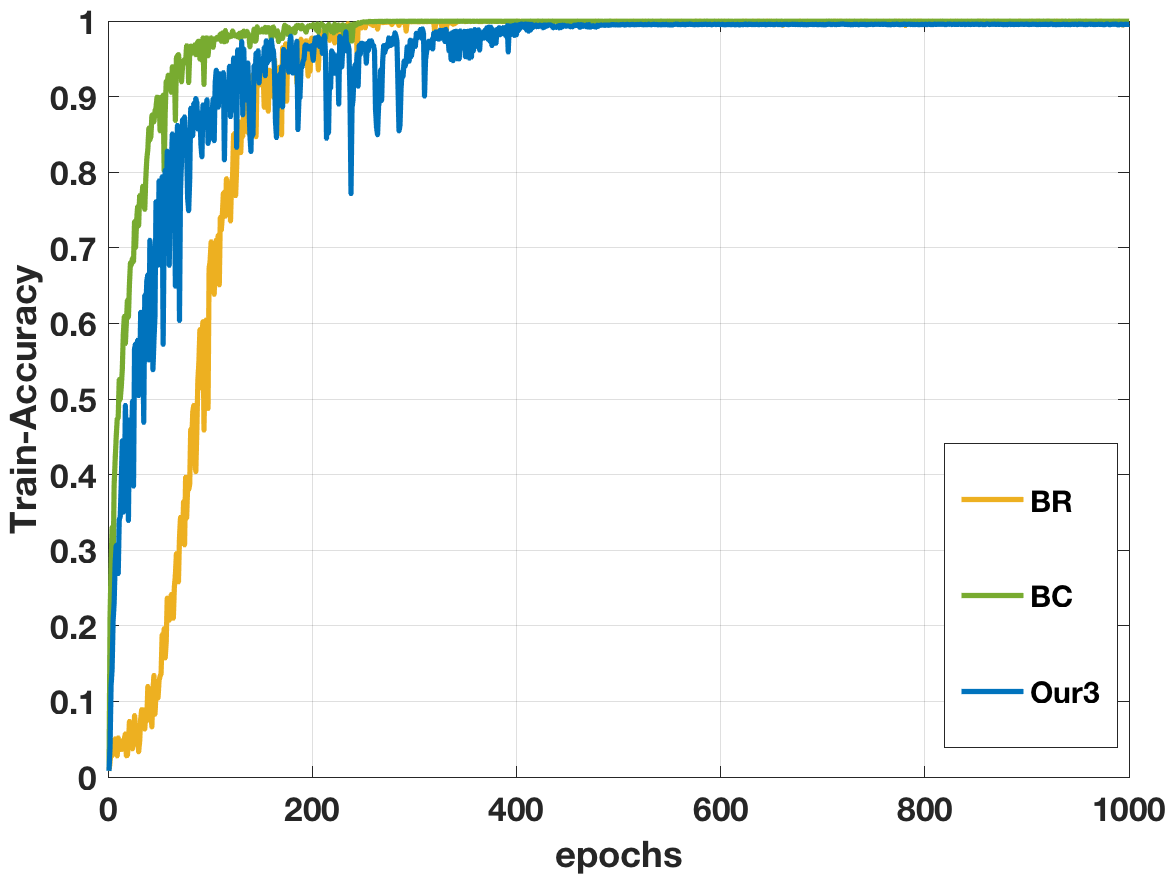} } 
	\hspace{0.1cm}		
	\subfloat[{\tt {\bf VGG-11 }}]{ \includegraphics[width=0.45\linewidth]{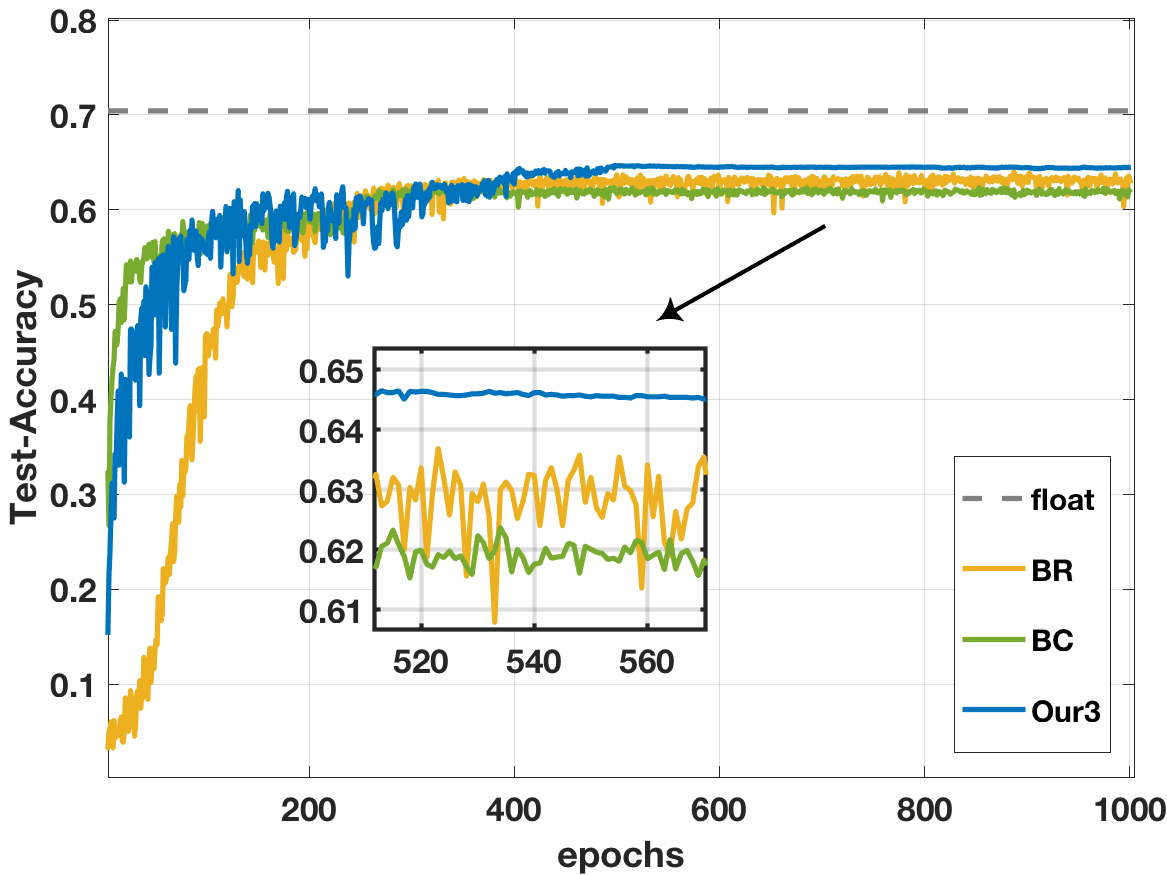} } 	\\[-2mm]	
	\hspace{0.1cm}	
	\subfloat[{\tt {\bf VGG-16 }}]{ \includegraphics[width=0.45\linewidth]{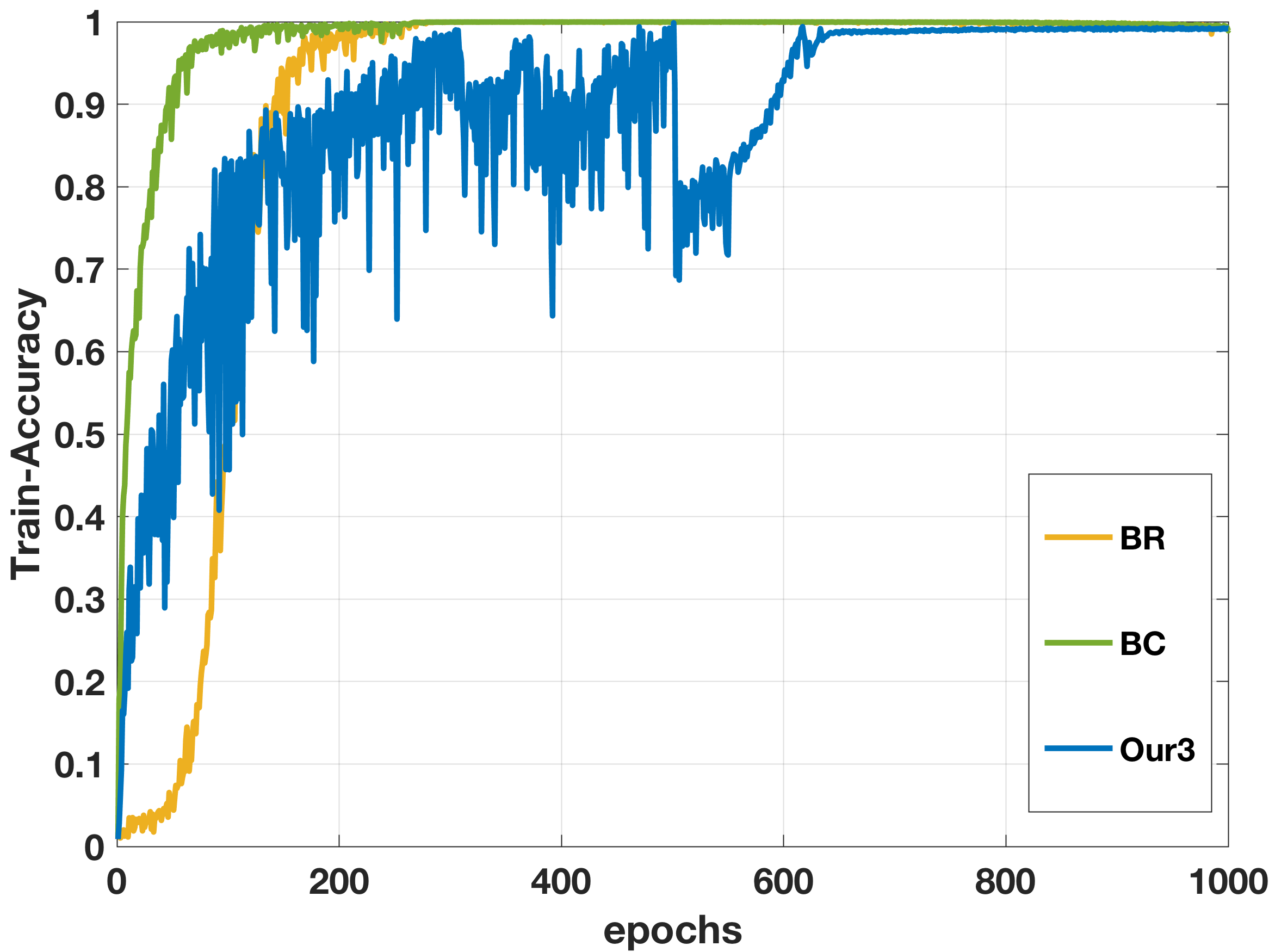} } 	
         \hspace{0.1cm}	
	\subfloat[{\tt {\bf VGG-16 }}]{ \includegraphics[width=0.45\linewidth]{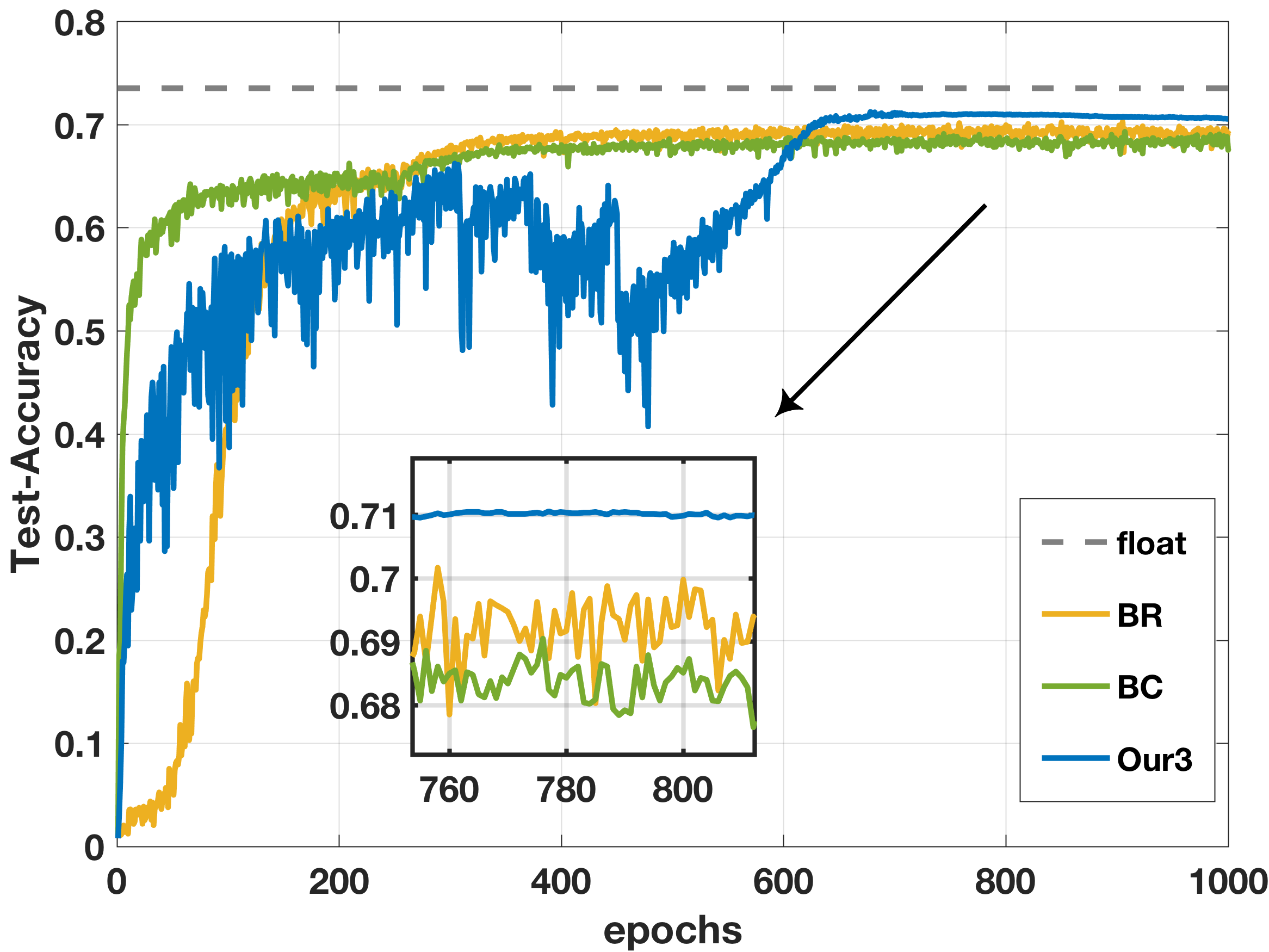} } \\[-2mm]	
		\hspace{0.1cm}	
	\subfloat[{\tt {\bf ResNet-18 }}]{ \includegraphics[width=0.45\linewidth]{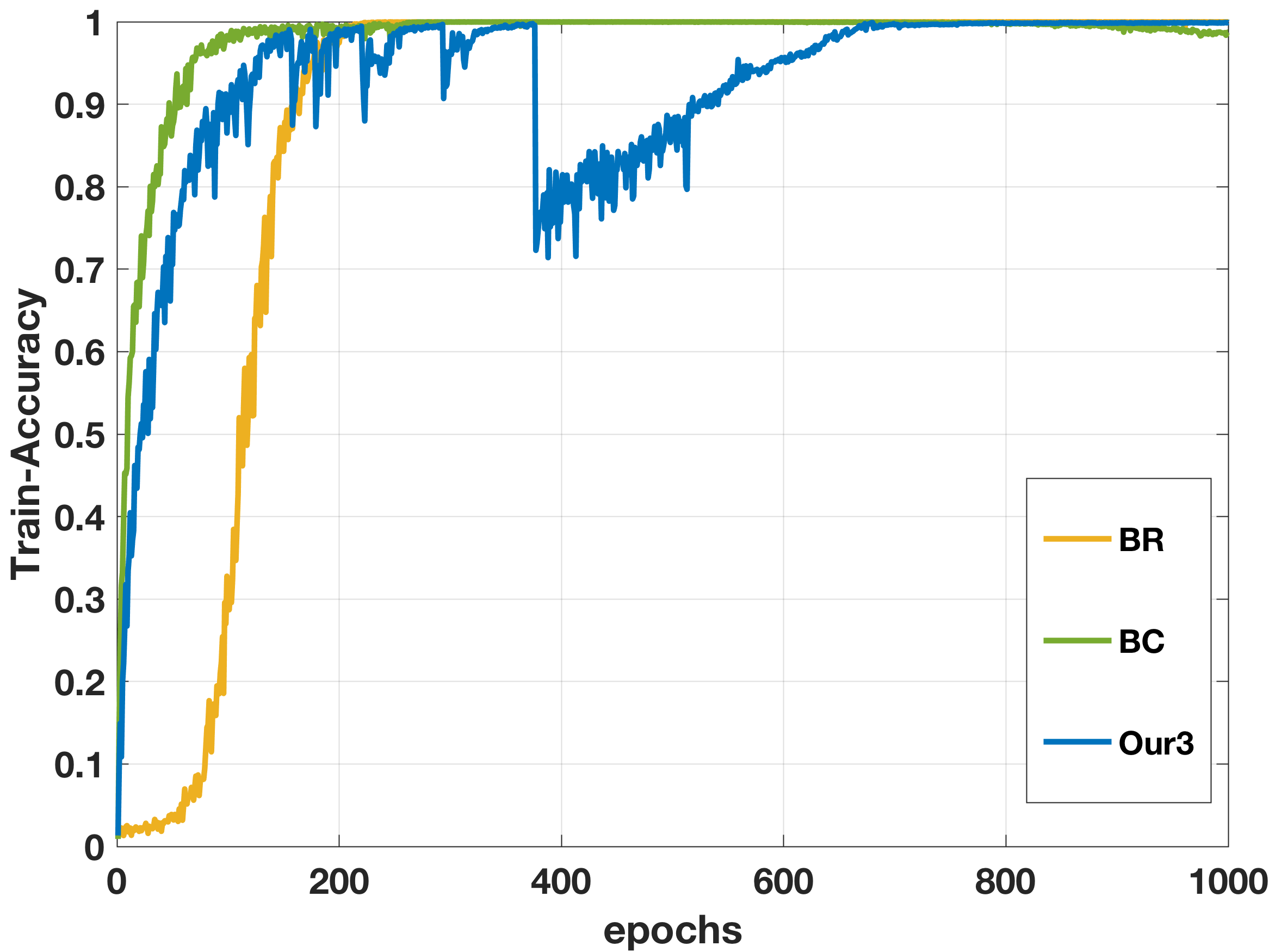} } 	
         \hspace{0.1cm}	
	\subfloat[{\tt {\bf ResNet-18 }}]{ \includegraphics[width=0.45\linewidth]{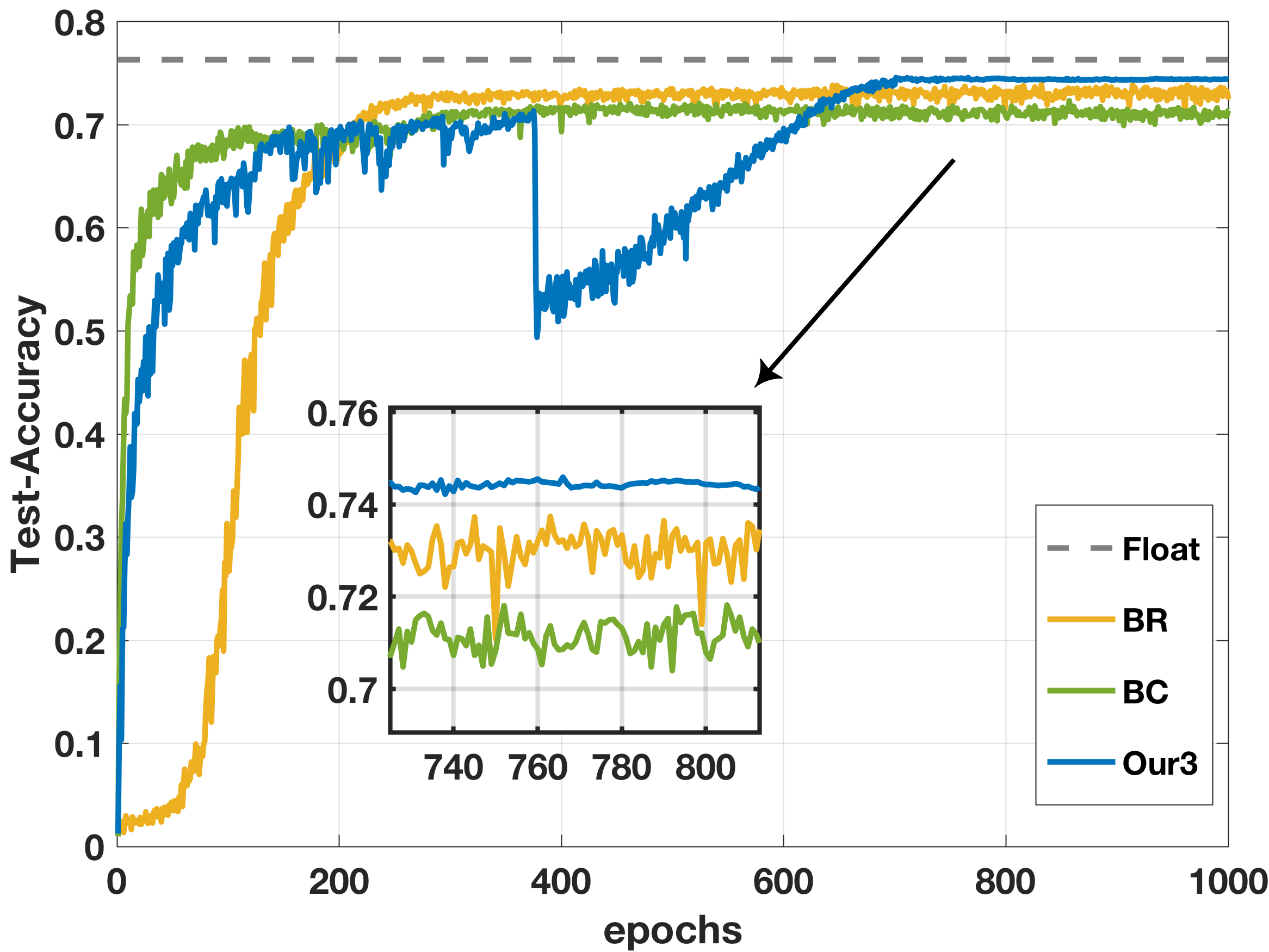} } 
	\caption{ Test accuracy and train accuracy for {\bf CIFAR-100} dataset.}
    \label{acc_cifar100}
\end{figure}

For CIFAR-100 dataset, the same three neural networks,  VGG-11, VGG-16 and ResNet-18 are investigated. The total epochs number are set as $1000$ and the batch-size is set to $128$ for all the neural networks. In the BR algorithm, we set the parameter $K = 200$ to start the second phase. For BC and BR algorithms,  their parameters are set as described in their papers \citep{CBD} and \citep{YZLOQX} respectively. In the previous subsection, we present  three sets of  parameters for our algorithm and it has been shown that the third set achieved the best performance. Therefore, we only show the results with the third set of parameter, Our3 with  a slight modification.   Specifically, in the first stage, we  use $350$ epochs and set $\gamma=10$, $\lambda=0.8$ and $\beta=1e3$. In the second stage, we set $\gamma = 3$ and $\lambda = 15$. Moreover, we decrease $\gamma$ and  $\beta$ by  $\gamma = \max\{0.99\gamma, 5e-3\}$ and $\beta = \max\{0.999\beta,920 \}$  in the second stage.

The compared results of the train and  test accuracy are shown in Figure \ref{acc_cifar100}. From Figure \ref{acc_cifar100},  it can be seen from this figure that the test accuracy of our algorithm is always higher than that  of the other two algorithms after 400 epochs. Similarly, for VGG-16 and ResNet-18, when $\gamma$ and $\lambda$ are reset in the second phase, the train accuracy and test accuracy are consistently higher than those of BC and BR algorithms. In particular, for VGG16, the test accuracy of our algorithm is $1\%$ higher than that of the other two algorithms.

\section{Conclusion}
\label{conclusion}
In this paper, we consider a three-block nonconvex minimizing problem including two separable terms with one being  a finite-sum, and a cross term. This type of structure is arisen naturally for quantized neural network training. We propose a general three-block splitting algorithm, namely STAM, and establish the convergence and convergence rate under a mild ES condition. The numerical experiments on diverse NN structures show that the proposed algorithm  is effective for both training and test accuracy.

\acks{ Xiaoqun Zhang was supported by National Natural Science Foundation of China (No. 12090024) and Shanghai Municipal Science and Technology Major Project 2021SHZDZX0102. Fengmiao Bian is also partly supported by outstanding PhD graduates development scholarship of Shanghai Jiao Tong University. We thank the Student Innovation Center at Shanghai Jiao Tong University for providing us the computing services.}

\vskip 0.2in



\begin{appendices}
\section{Proofs of Lemmas \ref{es-lemma}, \ref{ES-GH} and \ref{conv-lemma}.}
\label{app1}

Firstly, we recall the following lemma whose proof can be found in \citep{KR}. 
\begin{lemma}\label{Lip-lemma} 
Let the function $f$ be bounded from below by an infimum $f^{inf} \in \mathbb{R}$, differentiable, and $\nabla f$ is $L$-Lipschitz. Then for all $x \in \mathbb{R}^d$ we have
\begin{equation}
\| \nabla f(x) \|^2 \leq 2L \left( f(x) - f^{inf} \right).
\end{equation}
\end{lemma}

Now we give the proofs of Lemma \ref{es-lemma} and \ref{ES-GH}.  

\noindent{\bf Proof of Lemma \ref{es-lemma}. }  According to the definition of $\nabla G (\cdot)$ and using the convexity of the squared norm $\| \cdot \|^2$, we have
\begin{equation}\label{eq100}
\begin{aligned}
&\mathbb{E} \left[ \| \tilde{\nabla} G (y) + \nabla_y H(x, y) \|^2 \right]\\
&=\mathbb{E} \left[ \| \frac{1}{N} \sum_{i=1}^{N} v_i \nabla G_i(y) + \nabla_y H(x, y) \|^2 \right]\\
&\leq \mathbb{E} \left[ 2 \| \frac{1}{N} \sum_{i=1}^{N} v_i \nabla G_i (y) \|^2 + 2 \| \nabla_y H(x, y) \|^2 \right]\\
&\leq 2 \mathbb{E} \left\| \frac{1}{N} \sum_{i=1}^{N} v_i \nabla G_i (y)  \right\|^2 + 2 \mathbb{E} \| \nabla_y H(x, y) \|^2\\
&\leq  \frac{2}{N} \sum_{i=1}^{N} \mathbb{E} \| v_i \nabla G_i (y) \|^2 + 2 \mathbb{E} \| \nabla_y H(x,y) \|^2\\
&\leq \frac{2}{N} \sum_{i=1}^{N} \left[ \mathbb{E}\left[ v_i^2 \right] \| \nabla G_i(y) \|^2 + \|\nabla_y H(x,y)\|^2 \right].
\end{aligned}
\end{equation}
Furthermore, from Lemma \ref{Lip-lemma} we have
\begin{equation}
\| \nabla G_i (y) \|^2 \leq 2 L_1^i (G_i (y) - G_i^{inf}) ~~~\textmd{and}~~~ \| \nabla_y H(x, y) \|^2 \leq 2 L_3^* (H(x, y) - H^{inf}),
\end{equation}
where $G_i^{inf} = \inf_{y} G(y)$ and $H^{inf} = \inf_{x, y} H(x, y)$. Then we get
\begin{equation}\label{eq101}
\begin{aligned}
&\mathbb{E} \left[ \| \tilde{\nabla} G (y) + \nabla_y H(x, y) \|^2 \right]\\
&\leq \frac{4}{N}  \sum_{i=1}^{N} \left[ \mathbb{E}[v_i^2]  L_1^i (G_i (y) - G_i^{inf} ) +  L_3^* ( H(x, y) - H^{inf} ) \right]\\
&\leq \frac{4 \max_{i} (L_1^i \mathbb{E}[v_i^2] ) + L_3^*}{N} \sum_{i=1}^{N} \left( G_i (y) + H(x, y) - G_i^{inf} - H^{inf} \right)\\
&\leq \frac{4 \max_{i} (L_1^i \mathbb{E}[v_i^2] ) + L_3^*}{N} \sum_{i=1}^{N} \left( G_i (y) + H(x, y) - (G + H)^{inf} + (G+H)^{inf} - G_i^{inf} - H^{inf} \right)\\
&\leq \left(4 \max_{i} (L_1^i \mathbb{E}[v_i^2] + L_3^* \right) [ G(y) + H(x, y) - (G + H)^{inf} ] \\
&\quad + \frac{4 \max_{i} (L_1^i \mathbb{E}[v_i^2] ) + L_3^*}{N} \sum_{i=1}^{N} \left( (G+H)^{inf} - G_i^{inf} - H^{inf} \right).
\end{aligned}
\end{equation}
Note that $G (y) + H (x, y) = \frac{1}{N} \sum_{i=1}^{N} \left(G_i (y) + H(x,y) \right) \geq \frac{1}{N} \sum_{i=1}^{N}  \left( G_i ^{inf}+ H^{inf} \right)$, then $ \frac{1}{N} \sum_{i=1}^{N}  \left( G_i ^{inf}+ H^{inf} \right)$ is a lower bound of $G(y) + H(x,y)$, and then
\begin{equation}
\Delta^{inf} := \frac{1}{N} \sum_{i=1}^{N} \left( (G+H)^{inf} - G_i^{inf} - H^{inf} \right) = (G+H)^{inf} - \frac{1}{N} \sum_{i=1}^{N}  \left( G_i ^{inf}+ H^{inf} \right) \geq 0.
\end{equation}
\hfill\BlackBox \\

\noindent{\bf{Proof of Lemma \ref{ES-GH}.}} By  Cauchy inequality and ES inequality, we have
\begin{equation}
\begin{aligned}
&\mathbb{E} \left[ \| \tilde{\nabla} G(y) + \nabla_y H(x, y) \|^2 \right]\\
&\leq 2 \mathbb{E} \| \tilde{\nabla} G(y) \|^2 + 2 \| \nabla_y H(x, y) \|^2 \\
&\leq 4 A_0 \left( G(y) - G^{inf} \right) + 2 B_0 \| \nabla G(y) \|^2 + 2 C_0 + 2 \| \nabla_y H(x, y) \|^2.
\end{aligned}
\end{equation}
It follows from Lemma \ref{Lip-lemma} that
$$
\| \nabla G(y) \|^2 \leq 2 L_1 \left( G(y) - G^{inf} \right),
$$
$$
\| \nabla_y H(x, y) \|^2 \leq 2 L_3^* \left( H(x, y) - H^{inf} \right).
$$
Thus, we have
\begin{equation}
\begin{aligned}
&\mathbb{E} \left[ \| \tilde{\nabla} G(y) + \nabla_y H(x, y) \|^2 \right]\\
&\leq \left( 4 A_0 + 4 B_0 L_1 \right) \left(G(y) - G^{inf} \right) + 4 L_3^* \left( H(x, y) - H^{inf} \right) + 2 C_0\\
&\leq 2A \left[ G(y) + H(x, y) - (G + H)^{inf} \right] + 2A \left[ (G+H)^{inf} - G^{inf} - H^{inf} \right] + 2 C_0 \\
&=: 2A \left[ G(y) + H(x, y) - (G + H)^{inf} \right]  + C,
\end{aligned}
\end{equation}
where $A := \max \left( 2 A_0 +  2 B_0 L_1, 2 L_3^* \right)$ and $C: = 2A \left[ (G+H)^{inf} - G^{inf} - H^{inf} \right] + 2 C_0$. Thus we obtain the conclusion. 
\hfill\BlackBox \\

Next, we begin to prove Lemma \ref{conv-lemma}. Before proving Lemma \ref{conv-lemma}, we show an estimate for the gradient of objective function in model \eqref{model} with respect  to $y$-variable.

\begin{lemma}\label{par1}
Let Assumptions \ref{ass3} and \ref{ass2} be satisfied. Let $\beta, \gamma > 0$ be the parameters in Algorithm \ref{alg:STAM} and $A, C\geq 0$ be the constants in Lemma \ref{ES-GH}. Then, for any $M>0$ we have 
\begin{equation}\label{eq200}
\begin{aligned}
&\frac{1}{\beta} \eta_t + \frac{M}{2} \mathbb{E}\| y^{t+1} - y^t \|^2\\
&\leq \left( 1 + \frac{(L+M)A}{\beta^2}\right) \delta_t - \delta_{t+1} + \frac{(L+M)C}{2\beta^2} + \mathbb{E} \left[ H(x^{t+1}, y^{t+1}) - H(x^t, y^{t+1}) \right],
\end{aligned}
\end{equation}
where $\eta_t = \mathbb{E} \| \nabla G(y^t) + \nabla_y H(x^t, y^t) \|^2$ and $\delta_t = \mathbb{E} \left[ G(y^t) + H(x^t, y^t) - (G+H)^{inf} \right].$
\end{lemma}
\begin{proof}
By the optimality  condition of the first subproblem \eqref{algy}, we have
\begin{equation}\label{eq2}
y^{t+1} - y^t = \frac{- \tilde{\nabla} G(y^t) - \nabla_y H(x^t, y^t)}{\beta}.
\end{equation}
Using the L-smoothness of $G+H$ with respect to the variable $y$ in  Assumption \ref{ass2} $(a3)$, we get
\begin{equation}\label{eq3}
\begin{aligned}
&G(y^{t+1}) + H(x^t, y^{t+1})\\
&\leq G(y^t) + H(x^t, y^t) + \langle \nabla G(y^t) + \nabla_y H(x^t, y^t), y^{t+1} - y^t \rangle + \frac{L}{2} \| y^{t+1} - y^t \|^2.
\end{aligned}
\end{equation}
For any $M > 0$, adding up $\frac{M}{2} \| y^{t+1} - y^t \|^2$ on the both sides yields that 
\begin{equation}\label{eq4}
\begin{aligned}
&G(y^{t+1}) + H(x^t, y^{t+1}) + \frac{M}{2} \| y^{t+1} - y^t \|^2\\
&\leq G(y^t) + H(x^t, y^t) + \langle \nabla G(y^t) + \nabla_y H(x^t, y^t), y^{t+1} - y^t \rangle + \frac{L+M}{2} \| y^{t+1} - y^t \|^2\\
&\leq G(y^t) + H(x^t, y^t) + \langle \nabla G(y^t) + \nabla_y H(x^t, y^t), \frac{- \tilde{\nabla} G(y^t) - \nabla_y H(x^t, y^t)}{\beta} \rangle\\
&\quad ~+ \frac{L+M}{2} \| \frac{-\tilde{\nabla} G(y^t) - \nabla_y H(x^t, y^t)}{\beta}\|^2\\
&\leq G(y^t) + H(x^t, y^t) - \frac{1}{\beta} \langle  \nabla G(y^t) + \nabla_y H(x^t, y^t), \tilde{\nabla} G(y^t) + \nabla_y H(x^t, y^t) \rangle\\
&\quad ~+ \frac{L+M}{2\beta^2} \| \tilde{\nabla} G(y^t) + \nabla_y H(x^t, y^t)\|^2,
\end{aligned}
\end{equation}
where the second inequality follows from \eqref{eq2}. Taking expectations in \eqref{eq4} conditional on $y^t$, we obtain
\begin{equation}\label{eq5}
\begin{aligned}
&\mathbb{E} \left[ G(y^{t+1}) + H(x^t, y^{t+1}) + \frac{M}{2} \| y^{t+1} - y^t \|^2  ~ \bigg|  ~y^t \right]\\
&\leq G(y^t) + H(x^t, y^t) - \frac{1}{\beta} \| \nabla G(y^t) + \nabla_y H(x^t, y^t) \|^2 + \frac{L+M}{2\beta^2} \mathbb{E} \left[ \| \tilde{\nabla}G(y^t) + \nabla_y H(x^t, y^t)\|^2 \right].
\end{aligned}
\end{equation}
From Lemma \ref{ES-GH}, we have 
\begin{equation}\label{eq7}
\begin{aligned}
&\mathbb{E} \left[ G(y^{t+1}) + H(x^t, y^{t+1}) + \frac{M}{2} \| y^{t+1} - y^t \|^2 ~\bigg|~ y^t \right]\\
&\leq G(y^t) + H(x^t, y^t) - \frac{1}{\beta} \| \nabla G(y^t) + \nabla_y H(x^t, y^t) \|^2 \\
&\quad + \frac{(L+M)A}{\beta^2} \left[ G(y^t) + H(x^t, y^t) - (G+H)^{inf} \right] + \frac{(L+M)C}{2\beta^2},
\end{aligned}
\end{equation}
where $(G+H)^{inf}=\min_{x, y} G(y) + H(x, y)$. 
Subtracting $(G+H)^{inf}$ from both sides gives 
\begin{equation}\label{eq8}
\begin{aligned}
&\mathbb{E} \left[ G(y^{t+1}) + H(x^t, y^{t+1}) + \frac{M}{2} \| y^{t+1} - y^t \|^2 - (G+H)^{inf} ~\bigg|~ y^t \right]\\
&\leq \left(1 + \frac{(L+M)A}{\beta^2} \right) \left[ G(y^t) + H(x^t, y^t) - (G+H)^{inf} \right] + \frac{(L+M)C}{2\beta^2} \\
&\quad~ - \frac{1}{\beta}  \| \nabla G(y^t) + \nabla_y H(x^t, y^t) \|^2.
\end{aligned}
\end{equation}
Taking expectation again and applying the tower property, we obtain 
\begin{equation}\label{eq9}
\begin{aligned}
&\mathbb{E} \left[ G(y^{t+1}) + H(x^t, y^{t+1}) + \frac{M}{2} \| y^{t+1} - y^t \|^2 - (G+H)^{inf} \right]\\
&\leq \left( 1 + \frac{(L+M)A}{\beta^2}\right) \mathbb{E} \left[ G(y^t) + H(x^t, y^t) - (G+H)^{inf} \right] + \frac{(L+M)C}{2\beta^2} \\
&\quad - \frac{1}{\beta} \mathbb{E} \| \nabla G(y^t) + \nabla_y H(x^t, y^t) \|^2.
\end{aligned}
\end{equation}
Adding up $H(x^{t+1}, y^{t+1})$ on the both sides and rearranging,
\begin{equation}\label{eq10}
\begin{aligned}
&\mathbb{E} \left[ G(y^{t+1}) + H(x^{t+1}, y^{t+1}) - (G+H)^{inf} \right]  + \frac{1}{\beta} \mathbb{E} \| \nabla G(y^t) + \nabla_y H(x^t, y^t) \|^2\\
&\quad + \frac{M}{2} \mathbb{E} \| y^{t+1} - y^t \|^2\\
&\leq \left( 1 + \frac{(L+M)A }{\beta^2}\right) \mathbb{E} \left[ G(y^t) + H(x^t, y^t) -  (G+H)^{inf} \right] + \frac{(L+M)C}{2\beta^2}\\
&\qquad + \mathbb{E} \left[ H(x^{t+1}, y^{t+1}) - H(x^t, y^{t+1}) \right].
\end{aligned}
\end{equation}
Setting $\delta_{t+1} = \mathbb{E} \left[ G(y^{t+1}) + H(x^{t+1}, y^{t+1}) - (G+H)^{inf} \right]$ and $\eta_t = \mathbb{E} \| \nabla G(y^t) + \nabla_y H(x^t, y^t) \|^2$, we have
\begin{equation}\label{eq11}
\begin{aligned}
& \frac{1}{\beta} \eta_t + \frac{M}{2} \mathbb{E}\| y^{t+1} - y^t \|^2\\
&\leq \left( 1 + \frac{(L+M)A}{\beta^2}\right) \delta_t - \delta_{t+1} + \frac{(L+M)C}{2\beta^2} + \mathbb{E} \left[ H(x^{t+1}, y^{t+1}) - H(x^t, y^{t+1}) \right].
\end{aligned}
\end{equation}
\end{proof}

We next show an estimate for the gradient of objective function in model \eqref{model} with respect  to $x$-variable.

\begin{lemma}\label{par2}
Let Assumption \ref{ass2} be satisfied. Suppose that the parameter $\gamma > 0$ is chosen such that 
\[
\mathcal{K}_1 := -  \frac{-3 + 5\gamma l + 2\gamma}{2\gamma} - \frac{5 (1 + \gamma L_2^*)^2}{4 \gamma}  > 0. 
\]Then
\begin{equation}\label{eq250}
\begin{aligned}
& \mathbb{E} \left(\frac{\| z^t - z^{t-1} \|}{\gamma}\right)^2 \\
& \leq \frac{\mathcal{K}_2}{ \mathcal{K}_1} \bigg( \mathbb{E} \left[ H(x^t, y^{t+1}) - H(x^{t+1}, y^{t+1}) \right]  + \mathbb{E} \left[ \mathcal{M}_t - \mathcal{M}_{t+1} \right] + \mathcal{K}_3 \mathbb{E} \| y^{t+1} - y^t \|^2 \bigg).
\end{aligned}
\end{equation}
where 
$$
\mathcal{M}_t := \mathcal{M} (x^t, u^t, z^t) = F(u^{t}) + \frac{1}{2\gamma} \| 2x^{t} - u^{t} - z^{t} \|^2 - \frac{1}{2\gamma}\| x^{t} - z^{t} \|^2 - \frac{1}{\gamma} \| u^{t} - x^{t} \|^2 
$$
and 
$$
\mathcal{K}_2 := \frac{5(1 + \gamma L_2^*)^2}{4\gamma^2},  ~ ~ \mathcal{K}_3 := L_4^* ( 1 + 5\gamma L_4^* ) + \frac{5 \mathcal{K}_1 (L_4^*)^2}{\mathcal{K}_2}. 
$$

\end{lemma}

\begin{proof}
Since $x^{t+1}$ is the minimizer of subproblem \eqref{algx},  by the strongly convex of $H(\cdot, y^{t+1}) + \frac{1}{2\gamma} \| \cdot - z^t \|^2$  we have 
\begin{equation}\label{eq260}
\begin{aligned}
H(x^{t+1}, y^{t+1}) + \frac{1}{2\gamma} \| z^t - x^{t+1} \|^2 \leq H(x^{t}, y^{t+1}) + \frac{1}{2\gamma} \| z^t - x^t \|^2 - \frac{1}{2}\left(\frac{1}{\gamma} - l \right) \| x^{t+1} - x^t \|^2.
\end{aligned}
\end{equation}
On the other hand, using the fact that  $u^{t+1}$ is the minimizer of subproblem \eqref{algu}, we get 
\begin{equation}\label{eq261}
\begin{aligned}
&F(u^{t+1}) + \frac{1}{2\gamma} \| 2x^{t+1} - u^{t+1} - z^t \|^2\leq F(u^t) + \frac{1}{2\gamma} \| 2x^{t+1} - u^t - z^t \|^2.
\end{aligned}
\end{equation}
Summing up \eqref{eq260} and \eqref{eq261}, we have
\begin{equation}\label{eq262}
\begin{aligned}
&H(x^{t+1}, y^{t+1}) + F(u^{t+1}) + \frac{1}{2\gamma} \| 2x^{t+1} - u^{t+1} - z^t \|^2 + \frac{1}{2\gamma} \| z^t - x^{t+1} \|^2\\
&\leq H(x^t, y^{t+1}) + F(u^t) + \frac{1}{2\gamma} \| 2x^{t+1} - u^t - z^t \|^2 + \frac{1}{2\gamma} \| z^t - x^t \|^2 - \frac{1}{2}\left(\frac{1}{\gamma} - l \right) \| x^{t+1} - x^t \|^2.
\end{aligned}
\end{equation}
Notice that we can rewrite
\begin{equation}\label{eq263}
\begin{aligned}
&\| 2 x^{t+1} - u^{t+1} - z^t \|^2 - \| 2 x^{t+1} - u^{t+1} - z^{t+1} \|^2\\
&= 2 \langle 2 x^{t+1} - u^{t+1} - z^{t+1}, z^{t+1} - z^t \rangle + \| z^{t+1} - z^t \|^2\\
&=  2 \langle 2 x^{t+1} - u^{t+1}-(u^{t+1} - x^{t+1}) - z^{t}, z^{t+1} - z^t \rangle  + \| z^{t+1} - z^t \|^2\\
&=  - 4 \| z^{t+1} - z^t \|^2 + 2 \langle x^{t+1}- z^{t}, z^{t+1} - z^t \rangle + \| z^{t+1} - z^t \|^2\\
&=  - 2 \| z^{t+1} - z^t \|^2  - \| x^{t+1} - z^{t+1} \|^2 + \| x^{t+1} - z^t \|^2,
\end{aligned}
\end{equation}
where  the subproblem \eqref{algz} is used in the second and third equalities, and the fact that  $2\langle a, b \rangle = - \left( \| a - b \|^2 - \| a \|^2 - \| b \|^2\right)$ is used in the last equality.  Combining  \eqref{eq263} with \eqref{eq262} and using the subproblem \eqref{algz}, we obtain 
\begin{equation}\label{eq264}
\begin{aligned}
&H(x^{t+1}, y^{t+1}) + F(u^{t+1}) + \frac{1}{2\gamma} \| 2x^{t+1} - u^{t+1} - z^{t+1} \|^2 - \frac{1}{2\gamma}\| x^{t+1} - z^{t+1} \|^2 - \frac{1}{\gamma} \| u^{t+1} - x^{t+1} \|^2\\
&\leq H(x^t, y^{t+1}) + F(u^t) + \frac{1}{2\gamma} \| 2x^{t+1} - u^t - z^t \|^2 + \frac{1}{2\gamma} \| z^t - x^t \|^2 \\
&\quad - \frac{1}{\gamma} \| x^{t+1} - z^t \|^2 - \frac{1}{2}\left(\frac{1}{\gamma} - l \right) \| x^{t+1} - x^t \|^2.
\end{aligned}
\end{equation}
Using the basic equality $\| 2a - b -c \|^2 - 2 \| a - c \|^2 = 2 \| a -b \|^2 - \| b - c \|^2$ twice, we get 
\begin{equation}\label{eq650}
\begin{aligned}
&  \| 2x^{t+1} - u^t - z^t \|^2 +  \| z^t - x^t \|^2  \\
&= (2\| x^{t+1} - u^t \|^2 - \| u^t - z^t \|^2 + 2 \| x^{t+1} - z^t \|^2) + \| z^t - x^t \|^2\\
&= 2 \| x^{t+1} - u^t \|^2 +  2 \| x^{t+1} - z^t \|^2 + \| 2x^t - u^t - z^t \|^2 - \| x^t - z^t \|^2 - 2 \| x^t - u^t \|^2.
\end{aligned}
\end{equation}
Substituting \eqref{eq650} into \eqref{eq264} yields that   
\begin{equation}\label{eq265}
\begin{aligned}
&H(x^{t+1}, y^{t+1}) + F(u^{t+1}) + \frac{1}{2\gamma} \| 2x^{t+1} - u^{t+1} - z^{t+1} \|^2 - \frac{1}{2\gamma}\| x^{t+1} - z^{t+1} \|^2 - \frac{1}{\gamma} \| u^{t+1} - x^{t+1} \|^2\\
&\leq H(x^t, y^{t+1}) + F(u^t) + \frac{1}{2\gamma} \| 2x^{t} - u^{t} - z^{t} \|^2  - \frac{1}{2\gamma} \| x^{t} - z^t \|^2 - \frac{1}{\gamma} \| x^{t} - u^t \|^2  \\
&\quad + \frac{1}{\gamma} \| x^{t+1} - u^t \|^2 - \frac{1}{2}\left(\frac{1}{\gamma} - l \right) \| x^{t+1} - x^t \|^2. \\
\end{aligned}
\end{equation}
Furthermore,  the optimal condition of \eqref{algx} is given as 
\begin{equation}\label{eq13}
0 = \nabla_x H(x^{t+1}, y^{t+1}) + \frac{1}{\gamma} (x^{t+1} - z^t), 
\end{equation}
this implies that 
\begin{equation}\label{eq21}
\frac{1}{\gamma} (z^t - x^{t+1}) + l x^{t+1} = \nabla_x \left[H(\cdot, y^{t+1}) + \frac{l}{2} \| \cdot \|^2 \right] (x^{t+1}).
\end{equation}
Thus we also have 
\begin{equation}\label{eq22}
\frac{1}{\gamma} (z^{t-1} - x^t) + l x^t + \nabla_x H(x^t, y^{t+1}) - \nabla_x H(x^t, y^t) = \nabla_x \left[ H(\cdot, y^{t+1}) + \frac{l}{2} \| \cdot \|^2 \right] (x^t).
\end{equation}
Since the function $H(\cdot, y^{t+1}) + \frac{l}{2} \| \cdot \|^2 $ is a convex function, we get 
\begin{equation}\label{eq23}
\left\langle \frac{1}{\gamma} (z^t - x^{t+1}) + l x^{t+1} - \frac{1}{\gamma} (z^{t-1} - x^t) - l x^t - \left[ \nabla_x H(x^t, y^{t+1}) - \nabla_x H(x^t, y^t) \right], x^{t+1}-x^t \right \rangle \geq 0.
\end{equation}
Using the Cauchy inequality, we get 
\begin{equation}\label{eq24}
\frac{1}{\gamma} \langle z^t - z^{t-1}, x^{t+1} - x^t \rangle + \frac{1}{2} \| \nabla_x H(x^t, y^{t+1}) - \nabla_x H(x^t, y^t) \|^2 + \frac{1}{2}\| x^{t+1} - x^t \|^2 \geq \left( \frac{1}{\gamma}  - l \right) \| x^{t+1} - x^t \|^2, 
\end{equation}
and hence 
\begin{equation}\label{eq25}
\begin{aligned}
\frac{1}{\gamma} \langle z^t - z^{t-1}, x^{t+1} - x^t \rangle &\geq \left(\frac{1}{\gamma} - l - \frac{1}{2} \right) \| x^{t+1} - x^t \|^2 - \frac{1}{2} \| \nabla_x H(x^t, y^{t+1}) - \nabla_x H(x^t, y^t) \|^2\\
&\geq \left( \frac{1}{\gamma} - l - \frac{1}{2} \right) \| x^{t+1} - x^t \|^2 - \frac{L_4^{*}}{2} \| y^{t+1} - y^t \|^2,
\end{aligned}
\end{equation}
where  the Lipschitz condition of $\nabla_x H(x, y)$ is used in the last inequality. This is also equivalent to 
\begin{equation}\label{eq26}
\langle z^t - z^{t-1}, x^{t+1} - x^t \rangle \geq \left( 1 - \gamma l - \frac{\gamma}{2} \right) \| x^{t+1} - x^t \|^2 - \frac{\gamma L_4^{*}}{2} \| y^{t+1} - y^t \|^2.
\end{equation}
Therefore, we can obtain the estimate on $\| x^{t+1} - u^t \|^2$ as follows: 
\begin{equation}\label{eq27}
\begin{aligned}
\| x^{t+1} - u^t \|^2 &= \| x^{t+1} - x^t + x^t - u^t \|^2\\
&=\| x^{t+1} - x^t - (z^t - z^{t-1}) \|^2\\
&=\| x^{t+1} - x^t \|^2 - 2 \langle x^{t+1} - x^t, z^t - z^{t-1} \rangle + \| z^t - z^{t-1} \|^2\\
&\leq (-1 + 2\gamma l + \gamma) \| x^{t+1} - x^t \|^2 + \gamma L_4^{*} \| y^{t+1} - y^t \|^2 + \| z^{t} - z^{t-1} \|^2.
\end{aligned}
\end{equation}
Denote $\Phi(x^t, u^t, z^t; y^{t+1}) := H(x^{t}, y^{t+1}) + \mathcal{M} (x^t, u^t, z^t)$, where 
\begin{equation}\label{eq651}
 \mathcal{M}_t := \mathcal{M} (x^t, u^t, z^t) = F(u^{t}) + \frac{1}{2\gamma} \| 2x^{t} - u^{t} - z^{t} \|^2 - \frac{1}{2\gamma}\| x^{t} - z^{t} \|^2 - \frac{1}{\gamma} \| u^{t} - x^{t} \|^2.
 \end{equation}
 Then substituting \eqref{eq27} into \eqref{eq265}, we have
\begin{equation}\label{eq300}
\begin{aligned}
&\Phi(x^{t+1}, u^{t+1}, z^{t+1}; y^{t+1}) - \Phi(x^t, u^t, z^t; y^{t+1})  \\
& \leq \frac{-3 + 5\gamma l + 2\gamma}{2\gamma} \| x^{t+1} - x^t \|^2 +  L_4^{*} \| y^{t+1} - y^t \|^2 + \frac{1}{\gamma} \| z^t - z^{t-1} \|^2.
\end{aligned}
\end{equation}
Furthermore, according to the optimal condition \eqref{eq13} we have
\begin{equation}\label{eq15}
0 = \nabla_x H(x^{t+1}, y^{t+1}) - \nabla_x H(x^t, y^t) + \frac{1}{\gamma} (x^{t+1} - z^t) - \frac{1}{\gamma} (x^t - z^{t-1}).
\end{equation}
Rewrite \eqref{eq15} as 
\begin{equation}\label{eq31}
\frac{1}{\gamma} (z^t - z^{t-1}) = \nabla_x H(x^{t+1}, y^{t+1}) - \nabla_x H(x^t, y^t) + \frac{1}{\gamma} (x^{t+1} - x^t).
\end{equation}
Using the Lipschitz continuity of $\nabla_x H(x, y)$ in Assumption \ref{ass2}, we have
\begin{equation}\label{eq34}
\| z^{t} - z^{t-1} \| \leq (1 + \gamma L_2^{*}) \| x^{t+1} - x^t \| +  \gamma L_4^{*} \| y^{t+1} - y^t \|. 
\end{equation}
This together with the Cauchy inequality gives 
\begin{equation}\label{eq35}
\begin{aligned}
 \| z^t - z^{t-1} \|^2 &\leq \frac{5}{4} \left(1+\gamma L_2^* \right)^2 \| x^{t+1} - x^t \|^2 + 5 (\gamma L_4^*)^2 \| y^{t+1} - y^t \|^2.\\
\end{aligned}
\end{equation}
Combining this with \eqref{eq300}, we have
\begin{equation}\label{eq36}
\begin{aligned}
&\Phi(x^{t+1}, u^{t+1}, z^{t+1}; y^{t+1}) - \Phi(x^t, u^t, z^t; y^{t+1})  \\
&\leq \left[ \frac{-3 + 5\gamma l + 2\gamma}{2\gamma} + \frac{5(1 + \gamma L_2^*)^2}{4\gamma} \right] \| x^{t+1} - x^t \|^2 +  L_4^* ( 1 + 5\gamma L_4^* ) \| y^{t+1} - y^t \|^2\\
&=: - \mathcal{K}_1 \| x^{t+1} - x^t \|^2 +  L_4^* ( 1 + 5\gamma L_4^* ) \| y^{t+1} - y^t \|^2,
\end{aligned}
\end{equation}
where $\mathcal{K}_1 = - \left[ \frac{-3 + 5\gamma l + 2\gamma}{2\gamma} + \frac{5 (1 + \gamma L_2^*)^2}{4 \gamma} \right] $. It is clear that $\mathcal{K}_1 > 0$ when $\gamma$ is less than a computable thresholding. Thus we have
\begin{equation}\label{eq38}
\mathcal{K}_1 \| x^{t+1} - x^t \|^2 \leq \Phi(x^{t}, u^{t}, z^{t}; y^{t+1}) - \Phi(x^{t+1}, u^{t+1}, z^{t+1}; y^{t+1})  \\+ L_4^* ( 1 + 5\gamma L_4^* ) \| y^{t+1} - y^t \|^2.
\end{equation}
Denote $\mathcal{K}_2 := \frac{5(1 + \gamma L_2^*)^2}{4\gamma^2}$. From \eqref{eq35} and \eqref{eq38}, we have
\begin{equation}\label{eq39}
\begin{aligned}
&\frac{\mathcal{K}_1}{\mathcal{K}_2 } \left( \frac{1}{\gamma} \| z^t - z^{t-1} \| \right)^2 \\
&\leq \mathcal{K}_1 \| x^{t+1} - x^t \|^2 + \frac{5\mathcal{K}_1 (L_4^*)^2 }{\mathcal{K}_2 } \| y^{t+1} - y^t \|^2,\\
&\leq \Phi(x^{t}, u^{t}, z^{t}; y^{t+1}) - \Phi(x^{t+1}, u^{t+1}, z^{t+1}; y^{t+1}) + \left(  L_4^* ( 1 + 5\gamma L_4^* ) + \frac{5 \mathcal{K}_1 (L_4^*)^2}{\mathcal{K}_2} \right) \| y^{t+1} - y^t \|^2\\
&=: \Phi(x^{t}, u^{t}, z^{t}; y^{t+1}) - \Phi(x^{t+1}, u^{t+1}, z^{t+1}; y^{t+1}) + \mathcal{K}_3 \| y^{t+1} - y^t \|^2,
\end{aligned}
\end{equation}
where $\mathcal{K}_3 = \left(  L_4^* ( 1 + 5\gamma L_4^* ) + \frac{5 \mathcal{K}_1 (L_4^*)^2}{\mathcal{K}_2} \right)$. By the definition of $\Phi(x^{t}, u^{t}, z^{t}; y^{t+1})$, we have
\begin{equation}\label{eq43}
\begin{aligned}
&\frac{\mathcal{K}_1}{\mathcal{K}_2}\left( \frac{1}{\gamma} \| z^t - z^{t-1} \| \right)^2 \leq H(x^t, y^{t+1}) - H(x^{t+1}, y^{t+1}) + \mathcal{M}_t - \mathcal{M}_{t+1} + \mathcal{K}_3 \| y^{t+1} - y^t \|^2.
\end{aligned}
\end{equation}
Taking the expectation on both side, we get
\begin{equation}\label{eq44}
\begin{aligned}
& \mathbb{E} \left(\frac{\| z^t - z^{t-1} \|}{\gamma}\right)^2 \\
& \leq \frac{\mathcal{K}_2}{ \mathcal{K}_1} \bigg( \mathbb{E} \left[ H(x^t, y^{t+1}) - H(x^{t+1}, y^{t+1}) \right]  + \mathbb{E} \left[ \mathcal{M}_t - \mathcal{M}_{t+1} \right] + \mathcal{K}_3 \mathbb{E} \| y^{t+1} - y^t \|^2 \bigg).
\end{aligned}
\end{equation}
\end{proof}

Finally, combined Lemma \ref{par1} and Lemma \ref{par2},  we give the  proof of  Lemma \ref{conv-lemma}. \\

\noindent{\bf{Proof of Lemma \ref{conv-lemma}.}} Summing \eqref{eq200} and \eqref{eq44}, we have
\begin{equation}\label{eq45}
\begin{aligned}
&\frac{1}{\beta} \eta_t + \frac{M}{2} \mathbb{E} \| y^{t+1} - y^t \|^2 + \frac{\mathcal{K}_1}{\mathcal{K}_2} \mathbb{E}\left(  \frac{\| z^t - z^{t-1} \|}{\gamma} \right)^2\\
&\leq \left( 1 + \frac{(L+M)A}{\beta^2}\right) \delta_t - \delta_{t+1} + \frac{(L+M)C}{2\beta^2} + \mathbb{E} \left[ H(x^{t+1}, y^{t+1}) - H(x^t, y^{t+1})\right]\\
&\quad ~ + \mathbb{E} \left[ H(x^t, y^{t+1}) - H(x^{t+1}, y^{t+1}) \right] + \mathcal{K}_3 \mathbb{E} \| y^{t+1} - y^t \|^2  + \mathbb{E} \left[ \mathcal{M}_{t} - \mathcal{M}_{t+1}\right].
\end{aligned}
\end{equation}
Taking $M = 2\mathcal{K}_3$, we get
\begin{equation}\label{eq46}
\begin{aligned}
&\frac{1}{\beta} \eta_t + \frac{\mathcal{K}_1}{\mathcal{K}_2} \mathbb{E}\left(  \frac{\| z^t - z^{t-1} \| }{\gamma} \right)^2\\
&\leq \left( 1 + \frac{(L+M)A}{\beta^2}\right) \delta_t - \delta_{t+1} + \frac{(L+M)C}{2\beta^2}+ \mathbb{E} \left[ \mathcal{M}_{t} - \mathcal{M}_{t+1}\right]\\
&= \left( 1 + \frac{(L+M)A}{\beta^2}\right) \delta_t - \delta_{t+1} + \frac{(L+M)C}{2\beta^2}  + \mathbb{E} \left[ \mathcal{M}_{t}- \inf_{t \geq 0}\mathcal{M}_t \right] - \mathbb{E} \left[ \mathcal{M}_{t+1} - \inf_{t \geq 0} \mathcal{M}_t \right]\\
&= \left( 1 + \frac{(L+M)A}{\beta^2}\right) \delta_t - \delta_{t+1} + \mathcal{M}_{t}^{\prime} - \mathcal{M}_{t+1}^{\prime} + \frac{(L+M)C}{2\beta^2},
\end{aligned}
\end{equation}
where $\mathcal{M}_{t}^{\prime}  = \mathbb{E} \left[ \mathcal{M}_{t} - \inf_{t \geq 0} \mathcal{M}_t \right]$. For $\omega_{-1} > 0$, define $\omega_t = \frac{\omega_{t-1}}{\left(1 + \frac{(L+M)A}{\beta^2} \right)}$. Clearly that $\{ \omega_t \}_{t \geq -1} $ is a decreasing and positive sequence. Multiplying $\beta \omega_t$ on the both sides, we get
\begin{equation}\label{eq47}
\begin{aligned}
& \omega_t \eta_t + \frac{\beta\mathcal{K}_1}{\mathcal{K}_2} \omega_t\mathbb{E}\left(  \frac{\| z^t - z^{t-1} \|}{\gamma} \right)^2\\
&\leq \beta \left(1 + \frac{(L+M)A}{\beta^2} \right) \omega_t \delta_t  - \beta \omega_t \delta_{t+1}  + \frac{(L+M)C}{2\beta}\omega_t + \beta \omega_{t-1} \mathcal{M}_{t}^{\prime} - \beta \omega_t \mathcal{M}_{t+1}^{\prime} \\
&\leq \beta\omega_{t-1} \delta_t  - \beta \omega_t \delta_{t+1} + \frac{(L+M)C}{2\beta}\omega_t + \beta \omega_{t-1} \mathcal{M}_{t}^{\prime} - \beta \omega_t \mathcal{M}_{t+1}^{\prime}. \\
\end{aligned}
\end{equation}
Summing up both sides from  $t = 0$ to $t=T-1$ we have
\begin{equation}\label{eq48}
\begin{aligned}
&\sum_{t= 0}^{T-1} \omega_t \eta_t + \frac{\beta \mathcal{K}_1}{\mathcal{K}_2} \sum_{t=0}^{T-1} \omega_t \mathbb{E}\left(  \frac{\| z^t - z^{t-1}\|}{\gamma} \right)^2\\
&\leq \beta \omega_{-1} \delta_0 - \beta \omega_{T-1} \delta_{T} + \frac{(L+M)C}{2\beta} \sum_{t=0}^{T-1} \omega_t  + \beta \omega_{-1} \mathcal{M}_{0}^{\prime} - \beta \omega_{T-1} \mathcal{M}_{T}^{\prime} \\
&\leq \beta \omega_{-1} \delta_0 + \beta \omega_{-1} \mathcal{M}_{0}^{\prime}  + \frac{(L+M)C}{2\beta} \sum_{t=0}^{T-1} \omega_t.  \\
\end{aligned}
\end{equation}
This completes the proof of Lemma \ref{conv-lemma}.
\hfill\BlackBox \\

\section{Proofs of Theorems \ref{conv} and \ref{conv-rate}}\label{P=MT0}\label{app2}

Based on Lemma \ref{conv-lemma} and Algorithm \ref{alg:STAM}, we prove Theorem \ref{conv}. \\

\noindent{\bf Proof of Theorem \ref{conv}.}
We will make use of Lemma \ref{conv-lemma} to complete the proof.  Define $W_T = \sum_{t=0}^{T-1} \omega_t$. Dividing $W_T$ on the both sides of \eqref{ineq1}, we get
\begin{equation}\label{eq49}
\begin{aligned}
& \min_{0 \leq t \leq T-1} \eta_t + \frac{\beta \mathcal{K}_1}{\mathcal{K}_2}  \min_{0 \leq t \leq T-1} \mathbb{E} \left( \frac{\| z^t - z^{t-1} \|}{\gamma}\right)^2\\
&\leq \frac{1}{W_T} \bigg( \sum_{t=0}^{T-1} \omega_t \eta_t + \frac{\beta \mathcal{K}_1}{\mathcal{K}_2 } \sum_{t=0}^{T-1} \omega_t \mathbb{E} \left( \frac{\| z^t - z^{t-1} \|}{\gamma}\right)^2 \bigg)\\
&\leq \frac{\omega_{-1}}{W_T} \beta \delta_0 + \frac{\omega_{-1}}{W_T} \beta \mathcal{M}_{0}^{\prime} + \frac{(L+M)C}{2\beta}.
\end{aligned}
\end{equation}
It is easy to see that 
\begin{equation}\label{eq50}
W_T = \sum_{t=0}^{T-1} \omega_t \geq \sum_{t=0}^{T-1} \min_{0 \leq i \leq T-1} \omega_i = T \omega_{T-1} = \frac{T \omega_{-1}}{\left( 1 + \frac{(L+M)A}{\beta^2}\right)^{T}}.
\end{equation}
 Using this in \eqref{eq49} and the fact that  $\frac{\beta\mathcal{K}_1}{\mathcal{K}_2} \geq 2$, we have
\begin{equation}\label{eq51}
\begin{aligned}
& \min_{0 \leq t \leq T-1} \left[\eta_t + 2\mathbb{E} \left( \frac{\| z^t - z^{t-1} \|}{\gamma}\right)^2 \right]\\
&\leq \frac{\left(1 + \frac{(L+M)A}{\beta^2} \right)^T}{T} \big( \beta \delta_0 +  \beta \mathcal{M}_{0}^{\prime} \big) + \frac{(L+M)C}{2\beta},
\end{aligned}
\end{equation}
where $\eta_t = \mathbb{E} \| \nabla G(y^t) + \nabla_y H(x^t, y^t) \|^2$. By the Lipschitz continuity of $\nabla_y H(\cdot, y^t)$ and Cauchy inequality, we have 
\begin{equation}\label{eq400}
\| \nabla G(y^t) + \nabla_y H(u^t, y^t) \|^2 \leq 2  \| \nabla G(y^t) + \nabla_y H(x^t, y^t) \|^2 + 2(L_5^*)^2 \| z^t - z^{t-1} \|^2.
\end{equation}
On the other hand, note that the optimal condition of the subproblem \eqref{algu} is
\begin{equation}\label{eq16}
0 \in \partial F(u^t) + \frac{1}{\gamma} (u^t - 2x^t + z^{t-1}).
\end{equation}
Combined this with \eqref{eq13}  and \eqref{algz}, we obtain 
\begin{equation}\label{eq17}
\frac{1}{\gamma} (z^{t-1} - z^t) \in \nabla_x H(x^t, y^t) + \partial F(u^t). 
\end{equation}
Hence,
\begin{equation}\label{eq401}
\frac{1}{\gamma} (z^{t-1} - z^t) + \bigg( \nabla_x H(u^t, y^t) - \nabla_x H(x^t, y^t) \bigg) \in \nabla_x H(u^t, y^t) + \partial F(u^t).
\end{equation}
This implies that
\begin{equation}\label{eq402}
\textmd{dist}^2 (0, \nabla_x H(u^t, y^t) + \partial F(u^t))  \leq 2 \left( \frac{\| z^t - z^{t-1} \| }{\gamma}\right)^2 + 2 (L_2^*)^2 \| z^t - z^{t-1} \|^2.
\end{equation}
Combined this with \eqref{eq400}, we get
\begin{equation}\label{eq403}
\begin{aligned}
&\mathbb{E} \| \nabla G(y^t) + \nabla_y H(u^t, y^t) \|^2 + \mathbb{E}  \textmd{dist}^2 (0, \nabla_x H(u^t, y^t) + \partial F(u^t) ) \\
&\leq 2 \eta_t + 2 \mathbb{E} \left( \frac{\| z^t - z^{t-1} \|^2}{\gamma}\right)^2 + 2 \big(  (L_2^*)^2 + (L_5^*)^2 \big) \mathbb{E} \| z^t - z^{t-1} \|^2\\
&\leq 2 \eta_t + 4 \mathbb{E} \left(\frac{\| z^t - z^{t-1} \| }{\gamma} \right)^2
\end{aligned}
\end{equation}
due to $(L_2^*)^2 + (L_5^*)^2 \leq \frac{1}{\gamma^2}$. Thus, it follows from \eqref{eq51} that
\begin{equation}\label{eq52}
\begin{aligned}
&\min_{0 \leq t \leq T-1} \bigg[ \mathbb{E} \| \nabla G(y^t) + \nabla_y H(u^t, y^t) \|^2 + \mathbb{E}~\textmd{dist} \big( 0,  \nabla_x H(u^t, y^t) + \partial F(u^t) \big) \bigg] \\
&\leq 2 \min_{0 \leq t \leq T-1} \left[ \eta_t + 2 \mathbb{E} \left(\frac{\| z^t - z^{t-1} \| }{\gamma} \right)^2 \right]\\
&\leq 2 \frac{\left( 1+ \frac{(L+M)A}{\beta^2} \right)^{T}}{T} \beta ( \delta_0 + \mathcal{M}_{0}^{\prime}) + \frac{(L+M)C}{\beta}. 
\end{aligned}
\end{equation}
The proof of Theorem \ref{conv} is finished.  \hfill\BlackBox \\

As a consequence of Theorem \ref{conv}, we can easily obtain Theorem \ref{conv-rate}.\\

\noindent{\bf Proof of Theorem \ref{conv-rate}.} According to the fact that $1 + x \leq e^x (x \geq 0)$ and choosing $\beta > 0$ such that $\beta > \sqrt{(L+M)AT}$, we have
\begin{equation}\label{eq700}
\left(1 + \frac{(L+M)A}{\beta^2} \right)^T \leq \textmd{exp} \left(\frac{(L+M)AT}{\beta^2} \right) \leq  \textmd{exp}(1) \leq 3.
\end{equation}
It follows from  Theorem \ref{conv} that \eqref{eq52} holds if $\beta \geq \frac{2\mathcal{K}_2}{\mathcal{K}_1}$.  Thus, by  \eqref{eq700} and \eqref{eq52}  we obtain
\begin{equation}\label{eq701}
\begin{aligned}
&\min_{0 \leq t \leq T-1} \bigg[ \mathbb{E} \| \nabla G(y^t) + \nabla_y H(u^t, y^t) \|^2 + \mathbb{E}~\textmd{dist}^2 \big( 0,  \nabla_x H(u^t, y^t) + \partial F(u^t) \big) \bigg] \\
&\leq \frac{(L+M)C}{\beta} +  \frac{6\beta}{T} \left( \delta_0  +  \mathcal{M}_{0}^{\prime} \right).
\end{aligned}
\end{equation}
To make the RHS of \eqref{eq701} less than $\epsilon^2$, we could require that $\frac{(L+M)C}{\beta} \leq \frac{\epsilon^2}{2}$ and $\frac{6\beta}{T} \left( \delta_0  +  \mathcal{M}_{0}^{\prime} \right) \leq \frac{\epsilon^2}{2}$. Then we have
\begin{equation}\label{eq702}
 \frac{(L+M)C}{\beta} \leq \frac{\epsilon^2}{2} \Rightarrow  \beta \geq \frac{2C(L+M)}{\epsilon^2},
\end{equation}
and
\begin{equation}\label{eq703}
\frac{6\beta}{T} \left( \delta_0  +  \mathcal{M}_{0}^{\prime} \right) \leq \frac{\epsilon^2}{2} \Rightarrow T \geq \frac{12\beta (\delta_0 + M_0^{\prime})}{\epsilon^2}.
\end{equation}
Substituting three requirements on  $\beta$ into \eqref{eq703} yields that 
\begin{equation}\label{eq704}
T \geq \frac{12\sqrt{(L+M)AT}(\delta_0 + M_0^{\prime})}{\epsilon^2},~~T \geq \frac{24 \mathcal{K}_2 (\delta_0 + M_0^{\prime})}{\mathcal{K}_1\epsilon^2}~~~\textmd{and}~~~T\geq \frac{24C(L+M)(\delta_0 + M_0^{\prime})}{\epsilon^4}.
\end{equation}
By simplifying the form \eqref{eq704}, we get 
\begin{equation}\label{eq705}
T \geq \frac{12(L+M)(\delta_0 + M_0^{\prime})}{\epsilon^2} \max \Big\{ \frac{2C}{\epsilon^2}, \frac{12(\delta_0+ M_0^{\prime})A}{\epsilon^2}, \frac{2\mathcal{K}_2}{\mathcal{K}_1(L+M)} \Big\}.
\end{equation}
Thus, the desired conclusion is showed. 
\hfill\BlackBox \\
\end{appendices}


\end{document}